\documentclass[11pt]{amsart}

\usepackage[table]{xcolor}
\usepackage{amssymb}
\usepackage{amsmath}
\usepackage{amscd}
\usepackage{amsrefs}
\usepackage{mathrsfs}
\usepackage{faktor}
\usepackage{tikz}
\usepackage{tikz-cd}
\usepackage{float}

\usepackage{tikz}
\usetikzlibrary{decorations.markings}
\usetikzlibrary{arrows}
\usetikzlibrary{shapes.geometric}
\usetikzlibrary{decorations.pathreplacing,decorations.markings}
\usetikzlibrary{positioning, quotes}
\usepackage[utf8]{inputenc}
\usepackage[top=1in,bottom=1in,right=1in,left=1in]{geometry}
\usepackage{hyperref}
\usepackage{enumitem}
\usepackage{xspace}
\definecolor{lightgray}{gray}{0.9}
\newcommand{\ra}{\rightarrow}

\usepackage[capitalise]{cleveref}
\usepackage{comment}

\newtheorem{theorem}{Theorem}[section]
\newtheorem{lem}[theorem]{Lemma}
\newtheorem{proposition}[theorem]{Proposition}

\theoremstyle{definition}
\newtheorem{dfn}[theorem]{Definition}
\newtheorem{ex}[theorem]{Example}
\newtheorem{rmk}[theorem]{Remark}
\newtheorem{ntn}[theorem]{Notation}

\numberwithin{theorem}{section}

\newcommand{\cG}{\mathcal{G}}
\newcommand{\Oh}{\mathcal{O}}
\newcommand{\Winit}{W_{\mathrm{init}}}

\definecolor{darkgreen}{rgb}{0.0,0.7,0.0}
\newenvironment{LC}{\noindent\color{darkgreen} LC:}{}

\newenvironment{AL}{\noindent\color{red} AL:}{}

\renewenvironment{proof}[1][]{\begin{trivlist}
\item[\hskip \labelsep {\bfseries Proof  \def\temp{#1}\ifx\temp\empty  #1\else  (#1)\fi
}]}{\hfill\(\square\) \end{trivlist}}

\DeclareMathOperator{\BS}{BS}
\DeclareMathOperator{\id}{id}

\DeclareMathOperator{\im}{im}

\DeclareMathOperator{\End}{End}

\newcommand{\NSPACE}{\ensuremath{\mathsf{NSPACE}}}

\newcommand{\ns}{\mathsf{NSPACE}}

\title{Languages, groups and equations}

\author{Laura Ciobanu and Alex Levine}

\address{Department of Mathematics, Heriot-Watt University, and the Maxwell Institute for Mathematical Sciences,
Edinburgh EH14 4AS, UK}

\email{L.Ciobanu@hw.ac.uk}

\address{Department of Mathematics, Alan Turing Building, The University
of Manchester, Manchester M13 9PL, UK}

\email{alex.levine@manchester.ac.uk}


\begin{document}

\begin{abstract}
  The survey provides an overview of the work done in the last 10 years to characterise solutions to equations in groups in terms of formal languages. We begin with the work of
  Ciobanu, Diekert and Elder, who showed that solutions to systems of equations in
  free groups in terms of reduced words are expressible as EDT0L languages. We provide a sketch of
  their algorithm, and describe how the free group results extend to hyperbolic groups. The characterisation of solutions as EDT0L languages is very robust, and many group constructions preserve this, as shown by Levine.

  The most recent progress in the area has been made for groups without negative curvature, such as virtually abelian, the integral Heisenberg group, or the soluble Baumslag-Solitar groups, where the approaches to describing the solutions are different from the negative curvature groups. In virtually abelian groups the solutions sets are in fact rational, and one can obtain them as $m$-regular sets. In the Heisenberg group producing the solutions to a single equation reduces to understanding the solutions to quadratic Diophantine equations and uses number theoretic techniques. In the Baumslag-Solitar groups the methods are combinatorial, and focus on the interplay of normal forms to solve particular classes of equations.

  In conclusion, EDT0L languages give an effective and simple combinatorial characterisation of sets of seemingly high complexity in many important classes of groups.
\end{abstract}

\maketitle

\section{Introduction}

The aim of this survey is to show how formal languages can be used to describe
solutions to equations in groups. Equations are a generalisation of two of
Dehn's decision problems for groups: the word problem and the conjugacy problem.
An \textit{equation} in a group \(G\) is an identity \(w = 1\), where \(w\) is a
word over \(G\) together with a finite set of \textit{variables} and their
inverses. A \textit{solution} to an equation is an assignment of an element of
\(G\) to each variable, such that plugging this into \(w\) yields a word
equivalent to \(1\). Thus the conjugacy problem in a group \(G\) can be thought
of as solving the equation \(X^{-1} g X h^{-1} = 1\) for all \(g, \ h \in G\). Since the word and conjugacy problems are undecidable in general, solving equations in arbitrary groups will also be undecidable, and so an important question is for which classes of groups solving equations can be successful.

Finding algorithms to decide whether or not
equations in a variety of different classes of groups and monoids admit solutions has been an active area of research since the 1950s. The first
major positive results are due to Makanin, in the 1970s, when he
proved that it is decidable whether a finite system of equations in a free group or free monoid
is satisfiable \cites{Makanin_systems, Makanin_semigroups, Makanin_eqns_free_group}. Since then, Makanin's work has been extended to
show the decidability of the satisfiability of equations in hyperbolic
\cites{rips_sela, dahmani_guirardel} and certain relatively hyperbolic groups \cite{dahmani}, solvable Baumslag-Solitar groups
\cite{diophantine_metabelian_grps}, right-angled Artin groups \cite{DiekertMuscholl},
and more.

Whilst Makanin's work can determine if an equation admits a solution, it does
not describe the set of solutions. Razborov later created a method that allows
one to construct the solutions to systems of equations in a free group
\cites{Razborov_thesis, Razborov_english}. Since sets of solutions are
often infinite, there are multiple ways they can be represented, if at all. One
method in which solutions can be described is by expressing the set of solutions
as a language, and then defining a grammar or a machine generating the language. A formal language is any set of words over a finite alphabet, and one of the standard ways to categorise languages is the Chomsky hierarchy, which classifies them in terms of complexity into regular, context-free, context-sensitive and recursively enumerable. Since context-free languages are not able to describe these sets (see Section \ref{sec:sol_lang}) and context-sensitve are too general to capture the structure of the solutions in an optimal way, a different class is needed; this class is the focus of the survey.

\textbf{EDT0L languages and solutions.} In 2016, the first author, Diekert and Elder successfully employed formal languages to describe
the set of solutions to systems of equations in free groups
\cite{eqns_free_grps}. The class of languages used was the class of EDT0L
languages. Diekert and Elder generalised this to virtually free groups
\cite{DEijac}, and Diekert, Je\.{z} and Kufleitner extended it to right-angled
Artin groups \cite{EDT0L_RAAGs}. Hyperbolic groups \cites{eqns_hyp_grps, eqns_hyp_grps_conf}, virtually abelian groups \cite{VAEP} and virtually direct
products of hyperbolic groups \cite{EDT0L_extensions} followed later.

In the 1960s, Lindemayer introduced a collection of classes of languages called
\textit{L-systems}, which were originally used for the study of growth of
organisms. EDT0L (\textbf{E}xtended \textbf{D}eterministic
\textbf{T}able \textbf{0}-interaction \textbf{L}indenmayer) languages are one of the L-systems, and were introduced by
Rozenberg in 1973 \cite{ET0L_EDT0L_def_article}. L-systems, including EDT0L
languages, were studied intensively in computer science in the 1970s
and early 1980s; the most studied families of L-systems are subclasses of indexed languages, which although not historically part of the Chomsky hierarchy, fit nicely between context-free and context-sensitive (see Figure \ref{L_system_containments_intro_fig}). A description of solution sets as EDT0L languages
was known before only for quadratic word equations over a free monoid by \cite{FerteMarinSenizergues14}. Since the first author, Diekert and Elder's use of EDT0L languages to
study equations in free groups, a number of other works have used EDT0L
languages (or a similar class called ET0L languages)
\cites{bishop_elder_journal, EDT0L_permuations, appl_L_systems_GT} to describe sets in groups other than solutions to equations. ET0L
languages generalise both EDT0L and context-free languages. We omit
the definition here, as they are never used directly, although we
will refer to them
(see \cite{math_theory_L_systems} for the definition).
%Figure \ref{L_system_containments_intro_fig} shows
%how EDT0L and ET0L languages fit into the Chomsky hierarchy of languages.

When writing the solutions as EDT0L languages, an intrinsic part is to set up the convention on how to write these, as well as establish a way of expressing group elements in terms of words. Suppose $G$ is generated by a finite set $X$ and a normal form $\eta$, that is, a uniform way of expressing elements as words over $X$, is fixed (see Definition \ref{def_NF}).
Our overarching convention (with more flexibility in the case of virtually abelian groups) is that for an equation on $n$ variables in $G$, any solution $(g_1, \dots, g_n)$, with $g_i \in G$, will be converted into a tuple of normal forms $(g_1\eta, \dots, g_n\eta)$, and then into a single word $(g_1 \eta) \# \cdots \# (g_n \eta)$ over the alphabet $X \cup \{\# \}$, where $\#$ is a symbol not in $X$. Writing the solutions in terms of a natural normal form, and not as arbitrary words representing group elements, is one of aspects that makes the algorithms and results about equations challenging. We collect in Table \ref{table:results_free_hyp} the first results obtained in this area, in loosely chronological order.

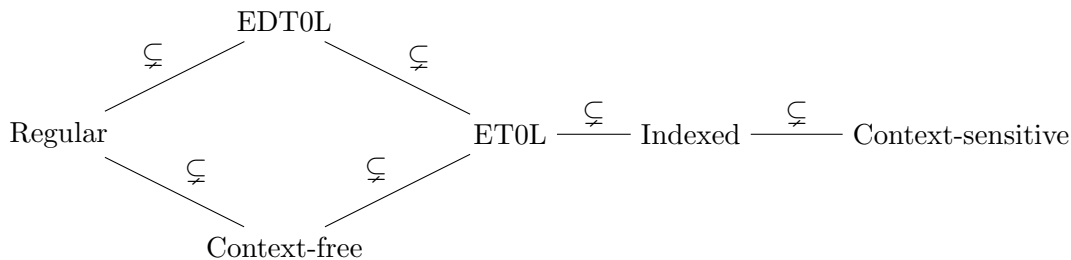
\begin{figure}
  \begin{tikzpicture}
    [scale=.6, auto=left]

    \node (reg) at (0, 0) {Regular};
    \node (EDT0L) at (5, 2.5) {EDT0L};
    \node (CF) at (5, -2.5) {Context-free};
    \node (ET0L) at (10, 0) {ET0L};
    \node (IND) at (14, 0) {Indexed};
     \node (CS) at (20, 0) {Context-sensitive};

    \draw (reg) edge["\(\subsetneq\)"] (EDT0L);
    \draw (reg) edge["\(\subsetneq\)"] (CF);
    \draw (EDT0L) edge["\(\subsetneq\)"] (ET0L);
    \draw (CF) edge["\(\subsetneq\)"] (ET0L);
    \draw (ET0L) edge["\(\subsetneq\)"] (IND);
   \draw (IND) edge["\(\subsetneq\)"] (CS);
  \end{tikzpicture}
    \caption{Reading left to right gives the containments of the classes of
  languages}
  \label{L_system_containments_intro_fig}
\end{figure}

% \begin{figure}
%   \label{L_system_containments_intro_fig}
%   \begin{tikzpicture}
%     [scale=.6, auto=left]

%     \node (reg) at (0, 0) {Regular};
%     \node (EDT0L) at (5, 2.5) {EDT0L};
%     \node (CF) at (5, -2.5) {Context-free};
%     \node (ET0L) at (10, 0) {ET0L};
%     \node (IND) at (14, 0) {Indexed};
%      \node (CS) at (20, 0) {Context-sensitive};

%     \draw (reg) to (EDT0L);
%     \draw (reg) to (CF);
%     \draw (EDT0L) to (ET0L);
%     \draw (CF) to (ET0L);
%     \draw (ET0L) to (IND);
%    \draw (IND) to (CS);
%   \end{tikzpicture}
%     \caption{Reading left to right gives the containments of the classes of
%   languages}
% \end{figure}

Many of the results have considered the space complexity in which the EDT0L
systems for solutions to equations can be constructed, where the input is the
size of the systems of equations, and in every case considered, this has been
shown to be in non-deterministic polynomial space. Note that an algorithm
that has non-deterministic \(\mathcal O(n)\) space complexity will have
non-deterministic \(\mathcal O(\exp(n))\) time complexity. We summarise the
results including space complexity (where \(n\) is the input length, which is
the total length of the system of equations) in
Tables~\ref{table:results_free_hyp},
\ref{table:results_constructions}, \ref{table:results_other} and
\ref{table:results_monoids}. For the constructions in
Table~\ref{table:results_constructions}, \(f\) and \(g\) refer to the functions
that are the non-deterministic space complexity of the groups used to build
them.
\begin{table}[ht]
\begin{center}

\resizebox{\columnwidth}{!}{%
\begin{tabular}{c|c|c|c|c|c}
Class of groups &generating set &solutions as & language  & $\NSPACE$ &\\
  \hline
\rowcolor{lightgray}
  Free & free basis & freely red words &  EDT0L & $n\log n$ &\cite{eqns_free_grps}\\
  Free &any & unique quasigeods  & EDT0L & $n\log n$ & \cite{eqns_hyp_grps}\\
\rowcolor{lightgray}
Free &any &  quasigeods  & ET0L & $n\log n$& \cite{eqns_hyp_grps}\\
Free &free basis & all words   & ET0L & $n\log n$& \cite{eqns_hyp_grps}\\
\rowcolor{lightgray}
Virtually free & certain & certain quasigeods  & EDT0L & $n^2\log n$ & \cite{DEijac}\\
Virtually free &any & unique quasigeods  & EDT0L & $n^2\log n$& \cite{eqns_hyp_grps}\\
\rowcolor{lightgray}
Virtually free &any & quasigeods & ET0L & $n^2\log n$& \cite{eqns_hyp_grps}\\
Torsion-free hyp &any & unique quasigeods & EDT0L & $n^2\log n$ &  \cite{eqns_hyp_grps}\\
\rowcolor{lightgray}
Torsion-free hyp &any & quasigeods  & ET0L & $n^2\log n$ &  \cite{eqns_hyp_grps}\\
 Hyp with torsion&any & unique quasigeods  & EDT0L & $n^4\log n$ &  \cite{eqns_hyp_grps} \\
\rowcolor{lightgray}
 Hyp with torsion &any & quasigeods  & ET0L & $n^4\log n$&  \cite{eqns_hyp_grps} \\

\end{tabular}%
}
\end{center}
\vspace{1cm}
\caption{Summary of results for free and hyperbolic groups}
\label{table:results_free_hyp}
\end{table}

\textbf{Free and hyperbolic groups.} The results in Table \ref{table:results_free_hyp} all refer to hyperbolic
groups, and these are indeed the groups where decidability of solving equations
was first established, by Makanin for free groups
\cites{Makanin_systems, Makanin_semigroups, Makanin_eqns_free_group}, then by Rips and Sela for
torsion-free hyperbolic \cite{rips_sela}, and by Dahmani and Guirardel \cite{dahmani_guirardel} for all hyperbolic groups.
The reason why hyperbolic groups have been tackled before other classes is,
roughly and intuitively, that the negative curvature helps set bounds on the
number and sizes of paths relating to thin triangles in hyperbolic Cayley
graphs: geodesic triangles are essential when solving systems of triangular
equations, and any equation in a hyperbolic group is handled by transforming it
into a system of triangular ones. Furthermore, hyperbolic groups are
particularly suitable for expressing the solutions in terms of normal forms, as
the most natural normal languages (all geodesics, quasi-geodesics, shortlex
normal forms, etc) are regular for any generating set. The geometry of
hyperbolic Cayley graphs, as well as the regularity of normal forms, were
exploited by the first author and Elder to a large degree in their work.

\textbf{Other classes of groups.} Outside the world of negative curvature, the main classes of groups considered have
been (virtually) abelian and nilpotent. Whilst a lot can be done for virtually
abelian groups in terms of both complexity of algorithms and formal languages
(see Section \ref{sec:virtab}), the results for nilpotent groups are very
limited because undecidability of solving systems of equations occurs already
for all non-abelian free nilpotent groups \cite{DuchinLiangShapiro}. However,
considering single one-variable equations in nilpotent groups such as the Heisenberg
group turned out to be fruitful, and led to some beautiful connections to
Diophantine equations of degree $2$ (see Section \ref{sec:Heis}). These results
are collected in Table~\ref{table:results_other}.
Even though the satisfiability of systems of equations in all free nilpotent groups is undecidable, for single equations there are some positive results.
Duchin, Liang and
Shapiro \cite{DuchinLiangShapiro} showed that the satisfiability of single
equations in all class \(2\) nilpotent groups with a virtually cyclic
commutator subgroup (including all Heisenberg groups) is decidable. Roman'kov
showed that this is not the case for all class \(2\) nilpotent groups
\cite{Romankov2016}. For groups with a higher nilpotency class, the
undecidability of solving even single equations occurs in free nilpotent groups
of class \(3\) and sufficiently large rank \cite{DuchinLiangShapiro}.

\begin{table}[ht]
\begin{center}

\resizebox{\columnwidth}{!}{%
%\rowcolors{1}{}{lightgray}
\begin{tabular}{c|c|c|c|c|c}

Class of groups &gen set &solutions as & language  & $\NSPACE$ &\\
  \hline
  \rowcolor{lightgray}
  Free abelian & free abelian basis & standard normal form &  EDT0L & $n^2$ &
  \cites{more_than_1700, VAEP} \\
  Virtually abelian & standard gen set & transversal normal form & EDT0L & $n^2$ &  \cite{VAEP} \\
  \rowcolor{lightgray}
  Heisenberg group & Mal'cev generators &
Mal'cev normal forms & EDT0L & $n^8 (\log n)^2$ & \cite{EDT0L_heisenberg}\\
  \rowcolor{lightgray}
(single 1-variable) & & & & &\\
BS(1,n) & standard gen set & see Section \ref{sec:BS}& EDT0L &  &
  \cite{DuncanEvettsHoltRees}\\
  (specific equations only) & & & & \\
  \rowcolor{lightgray}
RAAGs & standard gen set & standard normal form & EDT0L & $n \log n$ & \cite{EDT0L_RAAGs} \\
Dihedral Artin & inherited gen set  & inherited normal form & EDT0L & $n^2 \log n$ &  \cite{EDT0L_extensions}\\
& & (quasigeodesic) &&& \\
  \rowcolor{lightgray}
  Virtual direct products & inherited gen set  & inherited normal form  & EDT0L & $n^4 \log n$ & \cite{EDT0L_extensions} \\
  \rowcolor{lightgray}
of hyperbolic groups& & (quasigeodesic) & & &
\end{tabular}%
}
\end{center}
\vspace{1cm}
\caption{Summary of results for other groups}
\label{table:results_other}
\end{table}
%\begin{table}[ht]
%\begin{center}
%
%\resizebox{\columnwidth}{!}{%
%%\rowcolors{1}{}{lightgray}
%\begin{tabular}{r|r|r|r|r|l}
%
%Class of groups &gen set &solutions as & language  & $\NSPACE$ &\\
%  \hline
%  Free abelian & free abelian basis & standard normal form &  EDT0L & $n^2$ &
%  \cite{VAEP} (Volker?) \\
%  Virtually abelian & standard gen set & transversal noraml form & EDT0L & $n^2$ &  \cite{VAEP} \\
%  Heisenberg group & Mal'cev generators &
%Mal'cev normal forms & EDT0L & $n^8 (\log n)^2$ & \cite{EDT0L_heisenberg}\\
%(single 1-variable) & & &
%\end{tabular}%
%}
%\end{center}
%\vspace{1cm}
%\caption{Summary of results for virtually abelian and nilpotent groups}
%\label{table:results_VA_nilp}
%\end{table}

One of the most natural questions about solving equations in groups is whether decidability or the language complexity is preserved by group constructions. In recent work of the second author a number of positive results regarding the solution sets were obtained in this direction, and they are collected in Table \ref{table:results_constructions} and expanded on in Section \ref{sec:closure}.

\begin{table}[ht]
\begin{center}

\resizebox{\columnwidth}{!}{%
%\rowcolors{1}{}{lightgray}
\begin{tabular}{c|c|c|c|c|c}

Group construction&gen set &solutions as & language  & $\NSPACE$ &\\
  \hline
  \rowcolor{lightgray}
 Direct products & standard gen set & standard normal form &  EDT0L & \(f + g\)  &\cite{EDT0L_extensions}\\
Finite index subgroups & Schreier generators & Schreier normal form & EDT0L & $f$ & \cite{EDT0L_extensions}  \\
  \rowcolor{lightgray}
Wreath products  & standard gen set &  standard normal form & EDT0L & $f$& \cite{EDT0L_extensions} \\
  \rowcolor{lightgray}
with finite groups & & & & &\\
%Free &free basis & all words   & ET0L & $n\log n$& Cor.~\ref{cor:allwords}\\
%Virt free & certain & certain quasigeods  & EDT0L & $n^2\log n$ & \cite{DEijac}\\
%Virt free &any & unique quasigeods  & EDT0L & $n^2\log n$& Cor.~\ref{cor:changeGset}\\
%Virt free &any & quasigeods & ET0L & $n^2\log n$& Cor.~\ref{cor:changeGset}\\
%Torsion-free hyp &any & unique quasigeods & EDT0L & $n^2\log n$ &  Thm.~\ref{thmTorsionFree}\\
%Torsion-free hyp &any & quasigeods  & ET0L & $n^2\log n$ &  Thm.~\ref{thmTorsionFree}\\
% Hyp with torsion&any & unique quasigeods  & EDT0L & $n^4\log n$ &   Thm.~\ref{thmTorsion}\\
% Hyp with torsion &any & quasigeods  & ET0L & $n^4\log n$&   Thm.~\ref{thmTorsion}\\

\end{tabular}%
}
\end{center}
\vspace{1cm}
\caption{Summary of results for group constructions}
\label{table:results_constructions}
\end{table}

%Further classes of groups have been considered including right-angled Artin groups \cite{EDT0L_RAAGs}.
Using the constructions in
Table~\ref{table:results_constructions}, the second author
showed that solution languages to systems of equations
in groups that are virtually a direct product of
hyperbolic groups are EDT0L \cite{EDT0L_extensions}. As dihedral
Artin groups are virtually a direct product of free groups,
solutions to equations in these groups can be expressed as
EDT0L languages, although the normal form is non-standard.
Very recent work on the Baumslag-Solitar groups \cite{DuncanEvettsHoltRees} has
shown that certain classes of equations, including multiplication tables, can
be written as EDT0L languages with respect to a natural normal form. The authors
of this work conjecture that not all systems of equations with respect to this
normal form will have EDT0L solution languages, and thus this could be the
first example of a group with decidable Diophantine problem, but non-EDT0L
solution languages with respect to a natural regular normal form. These
results are collected in Table~\ref{table:results_other}.

When showing that solutions to systems of equations in free groups are EDT0L,
the authors first reduced the problem to showing that solutions to systems of
equations in free monoids with involution are EDT0L \cite{eqns_free_grps}. A similar approach was required for
right-angled Artin groups, reducing the problem to trace monoids \cite{EDT0L_RAAGs}. These
results for monoids are outlined in Table~\ref{table:results_monoids}.

\begin{table}[ht]
\begin{center}

\resizebox{\columnwidth}{!}{%
%\rowcolors{1}{}{lightgray}
\begin{tabular}{c|c|c|c|c|c}

Class of monoids &gen set &solutions as & language  & $\NSPACE$ &\\
  \hline
  \rowcolor{lightgray}
  Free monoids & standard basis & words &  EDT0L & $n\log n$ &
  \cite{eqns_free_grps}\\
  Trace monoids& standard gen set & standard normal form & EDT0L & $n \log n$ & \cite{EDT0L_RAAGs} \\
\end{tabular}%
}
\end{center}
\vspace{1cm}
\caption{Summary of results for monoids}
\label{table:results_monoids}
\end{table}

\textbf{Constraints.} The addition of constraints, that is, requiring the solutions to belong to specific sets, was also considered when showing solution
languages were EDT0L. Adding rational constraints allows one to restrict
specific variables to lie in certain sets, such as finitely generated
subgroups or conjugacy classes. Recognisable constraints are more
restrictive, but can be used to enforce that solutions lie in cosets
of finite index subgroups. In any class of groups where solution
languages to systems of equations are EDT0L, the same can be said for
systems of equations with recognisable constraints \cite{EDT0L_extensions}.
Rational constraints pose more challenges, however.

Systems of equations with rational constraints were shown to have EDT0L
solution langauges in free \cite{eqns_free_grps} and virtually free groups
\cite{DEijac}. For hyperbolic groups, the addition of a slightly more
restrictive class of constraints called quasi-isometrically embedded rational
constraints were shown to preserve the property of having EDT0L solutions when
added to systems of equations \cite{eqns_hyp_grps}. Rational constraints can
also be added to virtually abelian groups \cite{EDT0L_extensions}. For other
classes of groups, this may be difficult. For example, in the Heisenberg group
it remains open as to whether there is an algorithm to check if a given element lies
in a given rational set.

%\newpage
\tableofcontents

\section{Preliminaries}

  \begin{ntn}
    We will write functions to the right of their arguments, as \(xf\)
    or \((x)f\), rather than \(f(x)\). This means that when reading a
    long string of composed functions, we read left to right.
  \end{ntn}

  \subsection{Group equations}

  We define here a system of equations within a group.
  %Twisted equations prove useful in showing that systems of equations with rational constraints in
%  finite extensions of a group \(G\) have EDT0L solutions, if systems of
 % equations in \(G\) have EDT0L solutions.

  \begin{dfn}
		Let \(G\) be a group, and \(\mathcal X\) be a finite set of variables. A
		\textit{finite system of equations} in \(G\) with \textit{variables}
		\(\mathcal X\) is a finite subset \(\mathcal E\) of \(G \ast
		F(\mathcal{X})\), where \(F(\mathcal{X})\) is the free group on a finite set
		\(\mathcal{X}\). If \(\mathcal E = \{w_1, \ \ldots, \ w_n\}\), we view
		\(\mathcal E\) as a system by writing \(w_1 = w_2 = \cdots = w_n = 1\). A
		\textit{solution} to a system \(w_1 = \cdots = w_n = 1\) is a homomorphism
		\(\phi \colon F({\mathcal X}) \to G\), and such that \(w_1\bar{\phi} =
		\cdots = w_n \bar{\phi} = 1_G\), where \(\bar{\phi}\) is the extension of
		\(\phi\) to a homomorphism from \(G \ast F(\mathcal{X}) \to G\), defined by
		\(g \bar{\phi} = g\) for all \(g \in G\).

  \end{dfn}

  \begin{rmk}
    A solution to an equation with variables \(X_1, \ \ldots, \ X_n\) will
    usually be represented as a tuple \((x_1, \ \ldots, \ x_n)\) of group
    elements, rather than a homomorphism. We can obtain the homomorphism from
    the tuple by defining \(X_i \mapsto x_i\) for each \(i\).
  \end{rmk}

  We give some introductory examples of equations
  in groups.

  \begin{ex}
    Equations in \(\mathbb{Z}\) are linear equations in integers, and elementary
    linear algebra can be used to determine satisfiability, and also describe
    solutions.

    For example, if we write \(a\) as the free generator for \(\mathbb{Z}\),
    then
    \begin{equation}
      \label{Z_ex_eqn}
      a^3 X a^{-5} Y^{-1} a^{7} Y a^{-1} X = 1
    \end{equation}
    is an equation in \(\mathbb{Z}\). We can rewrite \eqref{Z_ex_eqn} using
    additive notation (and writing \(3\) rather
    than \(a^3\)) as
    \begin{equation}
      \label{Z_ex_eqn_2}
      3 + X - 5 - Y + 7 + Y - 1 + X = 0.
    \end{equation}
    Using the fact that \(\mathbb{Z}\) is abelian, we can manipulate
    \eqref{Z_ex_eqn_2} to obtain the following equation with the same set of
    solutions:
    \[
      2X + 4 = 0.
    \]
    Thus our set of solutions (when written as tuples) will be
    \[
      \{(2, \ y) \mid y \in \mathbb{Z}\}.
    \]
  \end{ex}

  \begin{ex}
    Checking if an element is a commutator can be done using an equation. In a
    group \(G\), the equation \(XYX^{-1}Y^{-1} = g\), where \(X\) and \(Y\) are
    variables and \(g \in G\), admits a solution if and only if \(g\) can be
    written as a commutator. This is an important \emph{quadratic} equation,
    and in free groups one can fully enumerate the solutions using Nielsen
    transformations (see for example \cite{GrigorchukKurchanov}).
  \end{ex}

  % \begin{ex}
  %   The conjugacy problem in any group can be viewed as an equation \(X^{-1} g X
  %   = h\), where \(g\) and \(h\) are group elements, and \(X\) is a variable.
  %   For example, in the free group \(F(a, \ b)\), one could consider the
  %   equation \(X^{-1} ab X = ba\). The set of solutions is \(\{(ab)^n b^{-1}
  %   \mid n \in \mathbb{Z}\}\).

  %   The twisted conjugacy problem can similarly be viewed as the equation
  %   \(X^{-1} g X = h \Phi\), for some automorphism \(\Phi\).
  % \end{ex}

  % \begin{ex}
  %   Let \(\Phi = \left(\begin{array}{cc} 0 & 1 \\ -1 & 0	 \end{array} \right)
  %   \in \GL_2(\mathbb{Z})\). Consider the twisted equation in \(\mathbb{Z}^2\),
  %   with the variables \(\mathbf X\) and \(\mathbf Y\):
  %   \[
  %     (\mathbf{X}) \Phi = \mathbf{Y}.
  %   \]
  %   This is just the automorphism problem in \(\mathbb{Z}^2\), which can be
  %   solved using elementary linear algebra.
  %   % The set of solutions will be \( \{(\mathbf{x}, \ \mathbf{x} A) \mid
  %   % \mathbf{x} \in \mathbb{Z}^2\}\).
  %   % Using basic linear algebra, one can show that this equals
  %   % \[
  %   % 	\{((x_1, \ x_2), \ (-x_2, \ x_1)) \mid (x_1, \ x_2) \in \mathbb{Z}^2\}.
  %   % \]
  %   In the free group \(F(a, \ b)\), an example of a twisted equation would be
  %   \(X (Y \psi) a Y = b X^{-1}\), for some \(\psi \in \Aut(F(a, \ b))\)
  %   although computing solutions to this is more difficult.
  % \end{ex}

  \subsection{Formal languages}
    We briefly introduce the concept of a language.

    \begin{dfn}
      An \textit{alphabet} is a finite set. A \textit{word}
      over an alphabet \(\Sigma\) is a finite sequence of
      elements of \(\Sigma\), usually denoted \(a_1 \cdots
      a_n\) rather than \((a_1, \ \ldots, \ a_n)\). A
      \textit{language} over \(\Sigma\) is any set of words
      over \(\Sigma\).
    \end{dfn}

    \begin{ex} \
    \begin{enumerate}
        \item

      The English language, that is, the set of words found in an English dictionary, is a formal language over the
      alphabet \(\{a, \ \ldots,  z\}\).

      \item The set of \textit{freely reduced} words
      over the alphabet \(\{a, \ b, \ a^{-1}, \ b^{-1}\}\) (that is, the set of words that
      contain no subword of the form \(aa^{-1}\),
      \(a^{-1}a\), \(bb^{-1}\) or \(b^{-1}b\))
      is also a language.
      \end{enumerate}
    \end{ex}
  %Further examples will be given when defining each class of languages used.
    \begin{ntn} \ \label{not:IC}
    \begin{enumerate}
      \item We use \(\varepsilon\) to denote the empty word over
      any alphabet;
      \item The set of all words over an alphabet \(\Sigma\)
      is denoted \(\Sigma^\ast\);
      \item If \(\Sigma\) is a finite subset of a group, we
      will sometimes use \(\Sigma^\pm = \Sigma \cup \{a^{-1} \mid a
      \in \Sigma\}\), especially when
      \(\Sigma\) is a finite generating set for a group, and
      then every group element can be represented by
      a word in \(\Sigma^\pm\). However, we will occasionally also state that a generating set is `inverse-closed' instead and not use $\pm$, to simplify the notation.
    \end{enumerate}
    \end{ntn}

  \subsection{Regular languages}
    We need the definition of regular languages throughout, including in
    the definition of EDT0L languages, and for types of constraint that can
    be added to systems of equations. We start with the definition.

    \begin{dfn}
      A \textit{finite-state automaton (FSA)} is a \(4\)-tuple \(\mathcal{A} = (\Sigma, \
      \mathcal{G}, \ q_0, \ F)\), where
      \begin{enumerate}
          \item \(\Sigma\) is a (finite) alphabet;
          \item \(\mathcal{G}\) is a finite directed graph, with edges
          labelled using \(\Sigma \cup \{\varepsilon\}\);
          \item \(q_0 \in V(\mathcal{G})\) is called the \textit{start
          state};
          \item \(F \subseteq V(\mathcal{G})\) is called the set of
          \textit{accept states}.
      \end{enumerate}
      A word \(w \in \Sigma^\ast\) is \textit{accepted} by \(\mathcal{A}\) if
      there is a path in \(\mathcal{G}\) from \(q_0\) to a point in \(F\),
      such that \(w\) is obtained by reading the labels of the edges when tracing the path. The \textit{language accepted} by \(\mathcal{A}\) is the
      set of all words accepted by \(\mathcal{A}\).

      A language \(L \subseteq \Sigma^\ast\) is called \textit{regular} if
      it accepted by some finite-state automaton.
    \end{dfn}

   % We give some examples of regular and non-regular languages.

    \begin{ex} \label{ex:reglang}\
    \begin{enumerate}
      \item The language \(\{a, \ b\}^\ast\) is regular. It is accepted
      by the automaton in Figure~\ref{fig:Last_reg}.

    \begin{figure}[H]
      \begin{center}
      \begin{tikzpicture}
        [scale=.6, auto=left,every node/.style={circle}]
        \tikzset{
        % style to apply some styles to each segment of a path
        on each segment/.style={
          decorate,
          decoration={
            show path construction,
            moveto code={},
            lineto code={
              \path [#1]
              (\tikzinputsegmentfirst) -- (\tikzinputsegmentlast);
            },
            curveto code={
              \path [#1] (\tikzinputsegmentfirst)
              .. controls
              (\tikzinputsegmentsupporta) and (\tikzinputsegmentsupportb)
              ..
              (\tikzinputsegmentlast);
            },
            closepath code={
              \path [#1]
              (\tikzinputsegmentfirst) -- (\tikzinputsegmentlast);
            },
          },
        },
        % style to add an arrow in the middle of a path
        mid arrow/.style={postaction={decorate,decoration={
              markings,
              mark=at position .5 with {\arrow[#1]{stealth}}
            }}},
      }

        \node[draw, double] (q0) at (0, 0) {\(q_0\)};

        \draw[postaction={on each segment={mid arrow}}]
        (q0) to [out=40, in=-40, distance=2cm] (q0);

        \draw[postaction={on each segment={mid arrow}}] (q0) to
        [out=-140, in=140, distance=2cm] (q0);

        \node (la) at (-2.1, 0) {\(a\)};
        \node (lb) at (2, 0) {\(b\)};
      \end{tikzpicture}
      \end{center}
      \caption{Finite-state automaton for \(\{a, \ b\}^\ast\)
      with start state \(q_0\) and accept state \(q_0\)}
      \label{fig:Last_reg}
    \end{figure}
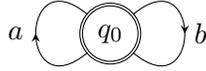
    \item The language \(\{w \in \{a, \ b,  \ a^{-1}, \
    b^{-1}\}^\ast \mid w
    \text{ is freely reduced}\}\) is regular. It is accepted
    by the finite-state automaton in Figure~\ref{fig:free_red}
        \begin{figure}[H]
      \begin{center}
      \begin{tikzpicture}
        [scale=1.1, auto=left,every node/.style={circle}]
        \tikzset{
        % style to apply some styles to each segment of a path
        on each segment/.style={
          decorate,
          decoration={
            show path construction,
            moveto code={},
            lineto code={
              \path [#1]
              (\tikzinputsegmentfirst) -- (\tikzinputsegmentlast);
            },
            curveto code={
              \path [#1] (\tikzinputsegmentfirst)
              .. controls
              (\tikzinputsegmentsupporta) and (\tikzinputsegmentsupportb)
              ..
              (\tikzinputsegmentlast);
            },
            closepath code={
              \path [#1]
              (\tikzinputsegmentfirst) -- (\tikzinputsegmentlast);
            },
          },
        },
        % style to add an arrow in the middle of a path
        mid arrow/.style={postaction={decorate,decoration={
              markings,
              mark=at position .5 with {\arrow[#1]{stealth}}
            }}},
      }

        \node[draw, double] (q0) at (0, 0) {\(q_0\)};
        \node[draw, double] (qa) at (2, 2) {\(q_a\)};
        \node[draw, double] (qA) at (-2, -2) {\(q_{a^{-1}}\)};
        \node[draw, double] (qb) at (2, -2) {\(q_b\)};
        \node[draw, double] (qB) at (-2, 2) {\(q_{b^{-1}}\)};

        \draw[postaction={on each segment={mid arrow}}]
        (q0) to node[midway, below]{\(a\)} (qa);
        \draw[postaction={on each segment={mid arrow}}]
        (q0) to node[midway, above]{\(a^{-1}\)} (qA);
        \draw[postaction={on each segment={mid arrow}}]
        (q0) to node[midway, above]{\(b\)} (qb);
        \draw[postaction={on each segment={mid arrow}}]
        (q0) to node[midway, below]{\(b^{-1}\)} (qB);

        \draw[postaction={on each segment={mid arrow}}]
        (qa) to [out=40, in=-40, distance=2cm]
        node[midway, right]{\(a\)} (qa);

        \draw[postaction={on each segment={mid arrow}}] (qA) to
        [out=-140, in=140, distance=2cm]
        node[midway, left]{\(a^{-1}\)} (qA);

        \draw[postaction={on each segment={mid arrow}}]
        (qb) to [out=40, in=-40, distance=2cm]
        node[midway, right]{\(b\)} (qb);

        \draw[postaction={on each segment={mid arrow}}] (qB) to
        [out=-140, in=140, distance=2cm]
        node[midway, left]{\(b^{-1}\)} (qB);

        \draw[postaction={on each segment={mid arrow}}]
        (qa) to [out=-60, in=60, distance=1cm]
        node[midway, right]{\(b\)} (qb);
        \draw[postaction={on each segment={mid arrow}}]
        (qb) to [out=120, in=-120, distance=1cm]
        node[midway, right]{\(a\)} (qa);

        \draw[postaction={on each segment={mid arrow}}]
        (qB) to [out=-65, in=65, distance=1cm]
        node[midway, left]{\(a^{-1}\)} (qA);
        \draw[postaction={on each segment={mid arrow}}]
        (qA) to [out=115, in=-115, distance=1cm]
        node[midway, left]{\(b^{-1}\)} (qB);

        \draw[postaction={on each segment={mid arrow}}]
        (qa) to [out=-150, in=-30, distance=1cm]
        node[midway, above]{\(b^{-1}\)} (qB);
        \draw[postaction={on each segment={mid arrow}}]
        (qB) to [out=30, in=150, distance=1cm]
        node[midway, above]{\(a\)} (qa);

        \draw[postaction={on each segment={mid arrow}}]
        (qb) to [out=-150, in=-30, distance=1cm]
        node[midway, below]{\(a^{-1}\)} (qA);
        \draw[postaction={on each segment={mid arrow}}]
        (qA) to [out=30, in=150, distance=1cm]
        node[midway, below]{\(b\)} (qb);
      \end{tikzpicture}
      \end{center}
      \caption{Finite-state automaton for \(\{w \in \{a, \ b,  \ a^{-1}, \
    b^{-1}\}^\ast \mid w
    \text{ is freely reduced}\}\) where \(q_0\) is the start
    state and every state is an accept state}
      \label{fig:free_red}
    \end{figure}
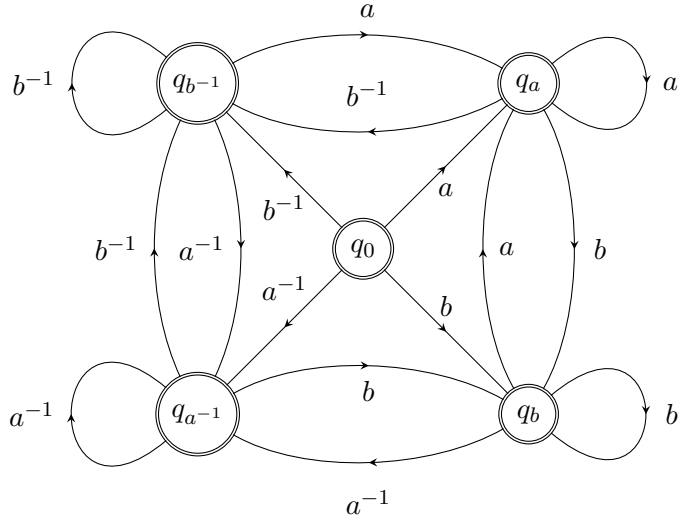
    \item The language \(\{a^n \# a^n \mid n \in \mathbb{Z}_{\geq 0}\}\) is not regular. Showing this
    requires the Pumping Lemma (see for example \cite{groups_langs_aut}).
    \end{enumerate}
    \end{ex}

 \subsection{Solution languages} \label{sec:sol_lang}
  We now explain how we represent solution sets as languages. We start by
  defining a normal form for elements in a group.

  \begin{dfn}\label{def_NF}
    Let \(G\) be a group, and \(\Sigma\) be a finite generating set for \(G\). A
    \textit{normal form} for \(G\), with respect to \(\Sigma\), is a function
    \(\eta \colon G \to (\Sigma^\pm)^\ast\) that fixes \(\Sigma^\pm\), and such
    that the word \(g \eta\) represents the element \(g\) for any \(g \in G\).

    A normal form \(\eta\) is called
    \begin{enumerate}
      \item \textit{regular} if \(\im \eta\) is a regular language over
      \(\Sigma^\pm\);
      \item \textit{geodesic} if \(\im \eta\) comprises only geodesic words
      in \(G\), with respect to \(\Sigma\); that is for all \(g \in G\),
      \(|g \eta| = |g|_{(G, \Sigma)}\);
      \item $(\lambda, \mu)$-\textit{quasi-geodesic} if there exist \(\lambda, \mu > 0\) such that
      \(|g \eta| \leq \lambda |g|_{(G, \Sigma)} + \mu\) for all \(g \in G\). We write \textit{quasi-geodesic} if the constants involved are not important.
    \end{enumerate}
  \end{dfn}

  Note that since $\eta$ is a function, the definition implies that each element will have a unique
  representative. We will often abuse notation and treat a
  normal form \(\eta\) as a subset \(\im \eta \subseteq (\Sigma^\pm)^\ast\), rather
  than as a function, as this is sometimes more convenient.

We give some standard examples of normal forms.
\begin{ex} \
\begin{enumerate}
    \item The standard normal form for $\mathbb{Z}^2$ on generators $\{a, b\}$ is the set $\{a^nb^m \mid m,n \in \mathbb{Z}\}$ over $\{a, a^{-1}, b, b^{-1}\}$. This normal form is geodesic and regular.
    \item The standard normal form for the free group $F(a,b)$ on generators $\{a, b\}$ is the set of reduced words over $\{a, a^{-1}, b, b^{-1}\}$. This is geodesic and regular, as shown in Example \ref{ex:reglang} (2).
    \item More generally, the set of shortlex representatives of a hyperbolic group $G$ over a generating set endowed with an order is a geodesic and regular normal form (see \cite{groups_langs_aut} for details).
    \item One of the standard normal forms used for elements in the Baumslag-Solitar groups $\BS(1,k)$ is given in Lemma~\ref{lem:BSnormform}. This normal form is regular but not geodesic.
\end{enumerate}
\end{ex}

  We can now express solutions as languages, with respect to
  a given normal form.

  \begin{dfn}
    Let \(G\) be a group with generating set \(\Sigma\),
    and let \(\eta \colon G \to (\Sigma^\pm)^\ast\) be a normal form for \(G\)
    with respect to \(\Sigma\). Let \(\mathcal{E}\) be a system of equations in
    \(G\) with a set \(\mathcal{S}\) of solutions. The \textit{solution
    language} to \(\mathcal{E}\) is the language
    \[
      \{(g_1 \eta) \# \cdots \# (g_n \eta) \mid (g_1, \ \ldots, \ g_n) \in
      \mathcal{S}\}
    \]
    over \(\Sigma^\pm \sqcup \{\#\}\).
  \end{dfn}

\begin{rmk}
Note that context-free languages do not in general work for describing solutions to equations when using the above convention. Even in
\(\mathbb{Z}\), the system $\{X = Y, Y = Z\}$ has the solution
language \(\{a^x \# a^x \# a^x \mid x \in \mathbb{Z}\}\) with respect to the
standard normal form, and this is not a context-free language.

\end{rmk}

\subsection{Space complexity}
  We briefly define space complexity. A more comprehensive introduction to complexity, including the standard big O notation $\Oh$, can be
  found in \cite{computational_compl}.

  \begin{dfn}
    Let \(f \colon \mathbb{Z}_{\geq 0} \to \mathbb{Z}_{\geq 0}\) be a function.
    We say that an algorithm runs in \(\mathsf{NSPACE}(f)\) if it can be
    performed by a non-deterministic Turing machine that satisfies the
    following:
    \begin{enumerate}
      \item A read-only input tape;
      \item A write-only output tape;
      \item A read-write work tape such that a computation path in the Turing
        machine with input length \(n\) uses at most \(\mathcal{O}(nf)\) units
        of the work tape.
    \end{enumerate}

    An algorithm is said to run in \textit{non-deterministic linear (resp.
    quadratic, resp. polynomial) space} if it runs in \(\mathsf{NSPACE}(f)\),
    for some linear (resp. quadratic, resp. polynomial) function \(f \colon
    \mathbb{Z}_{\geq 0} \to \mathbb{Z}_{\geq 0}\). %As with big O notation, we will write usually write \(\ns((n)f)\), rather than \(\ns(f)\).
  \end{dfn}

  % We will use this definition to show that we can construct the multivariable
  % finite-state automaton from Theorem \ref{theorem:VA_n_reg}, and hence the
  % EDT0L system from Corollary \ref{VA_EDT0L_cor}, in non-deterministic
  % quadratic space. Recall that a multivariable finite-state automaton has a
  % set of vertices, edges, an assignment of labels to edges, a specified start
  % state, and a set of accept states that all must be constructed.

  \begin{rmk}
    \label{constructible_rmk}
    We will often say that grammars or automata that define languages are
    \textit{constructible} in \(\ns(f)\). This means that there is an algorithm
    that runs in \(\ns(f)\), which takes an input that will be specified (which
    is sometimes other grammars or automata), and outputs the desired grammar or
    automaton.
  \end{rmk}

\subsection{Equations with constraints}

A \textit{finite system of equations with rational (recognisable)
		constraints} \(\mathcal{E}\) in a group \(G\) is a finite system of
		equations \(\mathcal{F}\) with variables \(X_1, \ \ldots, \ X_n\),
		together with a tuple of rational (recognisable) subsets \(R_1, \ \ldots, \
		R_n\) of \(G\). A \textit{solution} to \(\mathcal{E}\) is a solution
		\(\phi\) to \(\mathcal{F}\), such that \(X_i \phi \in R_i\) for all \(i\).

 We cover here the basic definitions of rational and recognisable subsets of
  monoids. Both types are used as constraints for variables in equations in
  groups, and we will use recognisable constraints to show that the class of
  groups where solutions to systems of equations form EDT0L languages is closed
  under passing to finite index subgroups. Whilst our main focus will be
  rational subsets of groups, rational subsets of monoids are
  required in the definition of an EDT0L language.

  %
  % Recall that a \textit{language} over \(\Sigma\) is any subset of
  % \(\Sigma^\ast\), where \(\Sigma\) is a finite set, called an
  % \textit{alphabet}. Recall also that a \textit{regular} language is any
  % language accepted by a finite-state automaton. We refer the reader to
  % \cite{groups_langs_aut}, Chapter 2, for further details on languages and finite
  % state automata.

% \lcc{Alex, can you also add rational and recognisable sets in groups? It's a bit confusing that you only mention monoids.}
  \begin{dfn}
    Let \(S\) be a monoid (for example, a group), and \(\Sigma\) be a (monoid)
    generating set for \(S\). Define \(\pi \colon \Sigma^\ast \to S\) to be the
    natural homomorphism. We say a subset \(A \subseteq S\) is
    \begin{enumerate}
      \item \textit{recognisable} if \(A \pi^{-1}\) is a regular language over
      \(\Sigma\);
      \item \textit{rational} if there is a regular language \(L\) over
      \(\Sigma\), such that \(A = L \pi\).
    \end{enumerate}
    Note that this definition is independent of the choice of generating set.
  \end{dfn}

  \begin{rmk}
    Recognisable sets are rational.
  \end{rmk}

 % We give a few examples of recognisable and rational sets.

  \begin{ex}
    Finite subsets of any monoid are rational. Finite subsets of a group \(G\)
    are recognisable if and only if \(G\) is finite \cite{herbst_thomas}. Finite
    index subgroups of any group are recognisable, and hence rational.
  \end{ex}

  % The following result of Grunschlag relates the rational subsets
  % of a finite index subgroup of a group \(G\) to the rational subsets of \(G\)
  % itself.

  % \begin{lem}[\cite{Grunschlag_thesis}, Corollary 2.3.8]
  %   \label{Grunschlag_rational_lem}
  %   Let \(G\) be a group with finite generating set \(\Sigma\), and \(H\) be a
  %   finite index subgroup of \(G\). Let \(\Delta\) be a finite generating set
  %   for \(H\), and \(T\) be a right transversal for \(H\) in \(G\). For each
  %   rational subset \(R \subseteq G\), such that \(R \subseteq H t\) for some
  %   \(t \in T\), there exists a (computable) rational
  %   subset \(S \subseteq H\) (with respect to \(\Delta\)), such that \(R = St\).
  % \end{lem}

  Herbst and Thomas proved that recognisable sets in a group \(G\) are always
  finite unions of cosets of a finite index normal subgroup of \(G\)
  \cite{herbst_thomas}. This is used to prove many facts about recognisable
  sets.

  % \begin{lem}
  %   \label{fin_index_recognisable_lem}
  %   Let \(G\) be a finitely generated group with a finite index subgroup \(H\),
  %   and let \(S \subseteq H\). Then \(S\) is recognisable in \(G\) if and only
  %   if \(S\) is recognisable in \(H\).
  % \end{lem}

  \section{EDT0L languages}

  The original definition is due to Rozenberg
  \cite{ET0L_EDT0L_def_article}, however, the use of the rational control, which
  often makes working with EDT0L languages much easier, is due to Asveld
  \cite{Asveld1977}.

  \begin{dfn}
    An \textit{EDT0L system} is a tuple \(\mathcal H = (\Sigma, \ C, \ \omega, \
    \mathcal{R})\), where
    \begin{enumerate}
      \item \(\Sigma\) is an alphabet, called the \textit{(terminal) alphabet};
      \item \(C\) is a finite superset of \(\Sigma\), called the
      \textit{extended alphabet} of \(\mathcal H\);
      \item \(\omega \in C^\ast\) is called the \textit{start (or seed) word};
      \item \(\mathcal{R}\) is a rational subset of
      \(\End(C^\ast)\), called the \textit{rational control} of \(\mathcal H\).
    \end{enumerate}
    The language \textit{accepted} by \(\mathcal H\) is
    \[
      L(\mathcal H) = \{\omega \phi \mid \phi \in \mathcal{R}\} \cap
      \Sigma^\ast.
    \]
    A language accepted by an EDT0L system is called an \textit{EDT0L
    language}.

    % We use \(\mathcal{EDT}0 \mathcal{L}\) to denote the class of EDT0L
    % languages.
    % The \textit{alphabet of the rational control} is a minimal set \(B\) of
    % labels on a finite-state automaton accepting \(\mathcal R\) (with respect to
    % cardinality). We say that \(\mathcal H\) is an \textit{EPDT0L system} if \(c
    % \phi \neq \varepsilon\) for all \(c \in C\) and \(\phi \in B\).
  \end{dfn}

  \begin{ntn}
    When defining endomorphisms of \(C^\ast\) for some extended alphabet \(C\),
    within the definition of an EDT0L system, we will usually define each
    endomorphism by where it maps each letter in \(C\). If any letter is not
    assigned an image within the definition of an endomorphism, we will say that
    it is fixed by that endomorphism.
  \end{ntn}

  We give some examples and non-examples of EDT0L languages.

  \begin{ex} \
  \begin{enumerate}
  \item The language \(L = \{a^{n^2} \mid n \in \mathbb{Z}_{> 0}\}\) is an EDT0L
    language over the alphabet \(\{a\}\). This can be seen by considering \(C =
    \{\perp, \ s, \ t, \ u, \ a\}\) as the extended alphabet of an EDT0L system
    accepting \(L\), with \(\perp\) as the start word, and using the finite
    state automaton from Figure \ref{n2_EDT0L_fig} to define the rational
    control.
    \begin{figure}[H]
      \begin{center}
      \begin{tikzpicture}
        [scale=.6, auto=left,every node/.style={circle}]
        \tikzset{
        % style to apply some styles to each segment of a path
        on each segment/.style={
          decorate,
          decoration={
            show path construction,
            moveto code={},
            lineto code={
              \path [#1]
              (\tikzinputsegmentfirst) -- (\tikzinputsegmentlast);
            },
            curveto code={
              \path [#1] (\tikzinputsegmentfirst)
              .. controls
              (\tikzinputsegmentsupporta) and (\tikzinputsegmentsupportb)
              ..
              (\tikzinputsegmentlast);
            },
            closepath code={
              \path [#1]
              (\tikzinputsegmentfirst) -- (\tikzinputsegmentlast);
            },
          },
        },
        % style to add an arrow in the middle of a path
        mid arrow/.style={postaction={decorate,decoration={
              markings,
              mark=at position .5 with {\arrow[#1]{stealth}}
            }}},
      }

        \node[draw] (q0) at (0, 0) {\(q_0\)};
        \node[draw] (q1) at (0, -5)  {\(q_1\)};
        \node[draw] (q2) at (5, -5) {\(q_2\)};
        \node[draw, double] (q3) at (0, -10) {\(q_3\)};

        \draw[postaction={on each segment={mid arrow}}] (q0) to (q1);

        \draw[postaction={on each segment={mid arrow}}] (q1) to
        [out=40, in=140, distance=1cm] (q2);

        \draw[postaction={on each segment={mid arrow}}] (q2) to
        [out=-140, in=-40, distance=1cm] (q1);

        \draw[postaction={on each segment={mid arrow}}] (q1) to (q3);

        \node (l1) at (-2.1, -2.35) {\(\varphi_\perp
        \colon \perp \mapsto tsa\)};

        \node (l2) at (2.5, -3.6) {\(\varphi_1 \colon s \mapsto su\)};

        \node (l3a) at (2.5, -6.4) {\(\varphi_2 \colon t \mapsto at\)};

        \node (l3b) at (3.05, -6.9) {\(u \mapsto ua^2\)};

        \node (l4) at (-2.2, -7.35) {\(\varphi_3 \colon
        s, t, u \mapsto \varepsilon\)};

      \end{tikzpicture}
      \end{center}
      \caption{Rational control for \(L = \{a^{n^2} \mid n \in \mathbb{Z}_{>
      0}\}\), with start state \(q_0\) and accept state \(q_3\).}
      \label{n2_EDT0L_fig}
    \end{figure}
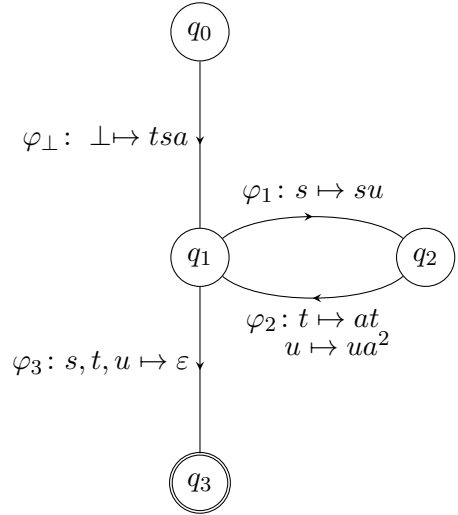
    \noindent The rational control can also be written as \(\varphi_\perp
    (\varphi_1 \varphi_2)^\ast \varphi_3\). This language is
    not context-free.

  \

 \item Consider the equation \(XY^{-1} = 1\) in \(\mathbb{Z}\) with the
    presentation \(\langle a \mid \rangle\). The solution language with
    respect to the standard normal form is
    \(
      L = \{a^n \# a^n \mid n \in \mathbb{Z}\},
    \)
    over the alphabet \(\{a, \ a^{-1}, \ \#\}\). The language \(L\) is EDT0L;
    our system will have the extended alphabet \(\{\perp, \ \#, \ a, \
    a^{-1}\}\) and rational control defined by Figure \ref{a^na^n_fig}.
    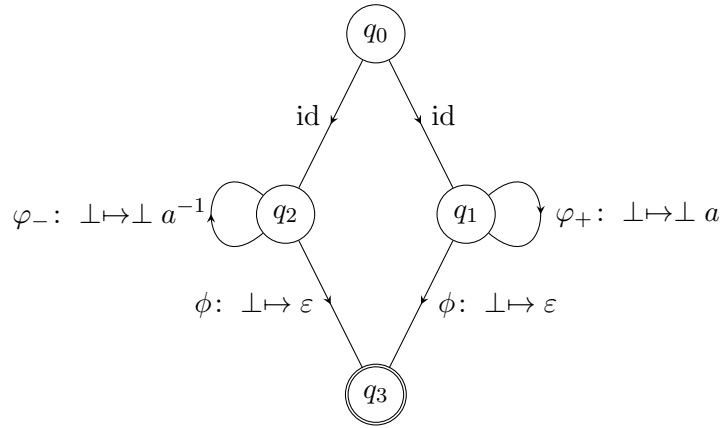
\begin{figure}[H]
      \begin{center}
      \begin{tikzpicture}
        [scale=.6, auto=left,every node/.style={circle}]
        \tikzset{
        % style to apply some styles to each segment of a path
        on each segment/.style={
          decorate,
          decoration={
            show path construction,
            moveto code={},
            lineto code={
              \path [#1]
              (\tikzinputsegmentfirst) -- (\tikzinputsegmentlast);
            },
            curveto code={
              \path [#1] (\tikzinputsegmentfirst)
              .. controls
              (\tikzinputsegmentsupporta) and (\tikzinputsegmentsupportb)
              ..
              (\tikzinputsegmentlast);
            },
            closepath code={
              \path [#1]
              (\tikzinputsegmentfirst) -- (\tikzinputsegmentlast);
            },
          },
        },
        % style to add an arrow in the middle of a path
        mid arrow/.style={postaction={decorate,decoration={
              markings,
              mark=at position .5 with {\arrow[#1]{stealth}}
            }}},
      }

        \node[draw] (q0) at (0, 0) {\(q_0\)};
        \node[draw] (q1) at (2, -4)  {\(q_1\)};
        \node[draw] (q2) at (-2, -4) {\(q_2\)};
        \node[draw, double] (q3) at (0, -8) {\(q_3\)};

        \draw[postaction={on each segment={mid arrow}}] (q0) to (q1);
        \draw[postaction={on each segment={mid arrow}}] (q0) to (q2);

        \draw[postaction={on each segment={mid arrow}}]
        (q1) to [out=40, in=-40, distance=2cm] (q1);

        \draw[postaction={on each segment={mid arrow}}] (q2) to
        [out=-140, in=140, distance=2cm] (q2);

        \draw[postaction={on each segment={mid arrow}}] (q1) to (q3);
        \draw[postaction={on each segment={mid arrow}}] (q2) to (q3);

        \node (l1) at (1.5, -1.8) {\(\id\)};
        \node (l2) at (-1.5, -1.8) {\(\id\)};
        \node (l3) at (5.8, -4) {\(\varphi_+ \colon \perp \mapsto \perp a\)};
        \node (l4) at (-5.9, -4) {\(\varphi_- \colon \perp \mapsto \perp
        a^{-1}\)};
        \node (l5) at (2.7, -6) {\(\phi \colon \perp \mapsto \varepsilon\)};
        \node (l5) at (-2.7, -6) {\(\phi \colon \perp \mapsto \varepsilon\)};

      \end{tikzpicture}
      \end{center}
            \caption{Rational control for \(L = \{a^n \# a^n \mid n \in \mathbb{Z}\}\)
      with start state \(q_0\), and accept state \(q_3\).}
      \label{a^na^n_fig}
    \end{figure}
    \noindent The rational control can also be expressed using the rational
    expression \(\{\varphi_-^\ast, \ \varphi_+^\ast\} \phi\).
This language is not regular, which can be shown using the pumping lemma
    (\cite{groups_langs_aut}, Theorem 2.5.17), but is context-free.

    \

    \item The equation \(XY^{-1} = 1\) in the free group $F(a,b)$ with
    generators $a,b$ has solution language with
    respect to the standard normal form, that is, freely reduced words, equal to
    \[
      L = \{w \# w \mid w \textrm{\ freely reduced over\ } \{a, a^{-1}, b, b^{-1}\}\}.
    \]
   The language \(L\) is EDT0L (similar to Example 1.3 of \cite{eqns_free_grps}), but not context-free. This generalises to any free group of finite rank.

   We construct an EDT0L system for \(L\). Our
   alphabet is \(\Sigma = \{a, \ b, \ a^{-1}, \ b^{-1}, \ \#\}\), our extended alphabet
   will be \(C = \Sigma \cup \{\perp\}\),
   and our start word will be \(\perp \# \perp\). For each
   \(x \in \{a, \ b, \ a^{-1}, \ b^{-1}\}\), define
   \(\varphi_x \in \End(C^\ast)\) by \(\perp \varphi_x =
   x \perp\). Let \(\mathcal{R} \subseteq \End(C^\ast)\) be
   the rational set obtained by replacing every \(x \in
   \{a, \ b, \ a^{-1}, \ b^{-1}\}\) labelling an edge in
   the finite-state automaton within Figure~\ref{fig:free_red}
   with \(\varphi_x\). Doing so we obtain
   \[
     \{(\perp \# \perp) \phi \mid \phi \in \mathcal{R}\}
     = \{w \perp \# w \perp \mid w \text{ freely reduced over }
     \{a, \ b, \ a^{-1}, \ b^{-1}\}\}.
   \]
   By post-concatenating the rational control with the map
   \(\psi \in \End(C^\ast)\) defined by \(\perp \psi =
   \varepsilon\), which can be done by adding a new accept
   state, making it the only accept state and adding a
   \(\psi\)-labelled edge from every former accept state to
   the new one, we obtain a new rational set \(\mathcal{R}'
   \subseteq \End(C^\ast)\) such that
      \[
     L = \{(\perp \# \perp) \phi \mid \phi \in \mathcal{R}'\}.
   \]
   \item The language
   \[
    \{w \in \{a, \ b, \ a^{-1}, \ b^{-1}\} \mid w=_{F(a, \ b)} 1\}
   \]
   is context-free, but not EDT0L \cite{appl_L_systems_GT}.
   \item
   The language
   \[
   \{w \in \{a, \ b\}^\ast \mid |w| = n^2 \text{
   for some } n \in \mathbb{Z}_{\geq 0}\}
   \]
   is not EDT0L (Corollary IV.3.4 of \cite{math_theory_L_systems}). It is not
   difficult to show that it is ET0L, but not context-free.
   \end{enumerate}
  \end{ex}

  \begin{theorem}[\cite{EDT0L_extensions}, Lemma 2.15]
    \label{EDT0L_closure_properties_thm}
    The class of EDT0L languages is closed under the following operations:
    \begin{enumerate}
      \item Finite unions;
      \item Intersection with regular languages;
      \item Concatenation;
      \item Kleene star closure;
      \item Image under free monoid homomorphisms.
    \end{enumerate}
%    The class of ET0L languages is closed under the above operations, together
%    with:
%    \begin{enumerate}
%      \item[(6)] Pre-image under free monoid homomorphisms.
%    \end{enumerate}
    Applying any of the operations will not affect the space complexity that
    systems for the languages involved can be constructed in, assuming that the
    regular language in (2) and the homomorphism in (5) are
    constructible in constant space.
  \end{theorem}

\section{Getting the solutions in free and hyperbolic groups}

In this section we give a sketch of the algorithms producing the EDT0L description in free and in hyperbolic groups. The main algorithm we describe is that for free groups, which is then used to solve equations in torsion-free hyperbolic groups. %We focus here on free and torsion-free hyperbolic groups.
There is a similar, but more involved, reduction from hyperbolic groups with torsion to virtually free groups. Since torsion adds a significant layer of difficulty to the algorithms and arguments involved, in order to keep this survey at an accessible level, we refer the reader to \cite{DEijac} and \cite{eqns_hyp_grps} for details on the torsion case.
In the same vein, we restrict ourselves to equations in groups, and do not include inequations or constraints in our general discussion in order to keep the exposition simpler.

\textbf{Overview of the algorithm.} At the core of the algorithm producing the solutions to an equation in a free group or monoid is the construction of a finite labeled graph: this graph, or automaton, has (i) verticesvertices  labeled by equations of bounded size which are variations of the input equation,
plus some additional data, and (ii) directed edges corresponding to  transformations applied to the equations; the edges are labelled by endomorphisms of a free monoid $C^*$, where $C$ is a finite alphabet which includes the  group or monoid generators. The graph, viewed as a (non-deterministic) finite-state automaton (FSA), produces a rational language of endomorphisms of $C^*$. In \cite{eqns_free_grps} we show that the set of all such endomorphisms applied to a particular `seed' word gives the full set of solutions to the input equation as reduced words. Thus, by the definition of Asveld \cite{Asveld1977}, we obtain that the solution set is an EDT0L language, and therefore an indexed language.
Moreover, one can decide if there are zero, infinitely or finitely many solutions simply by checking if the graph is empty.

We present below the free group algorithm and refer the reader to \cite{eqns_free_grps} for the free monoid case, which is similar (requires different notation, but not a different approach) to the free group case.

%\subsection{Free monoids and groups.}
\subsection{Sketch of the algorithm for free groups.}\label{sec:freegp}

\medskip

The following is a summary of the main steps used for getting the solutions to equations in free groups. The reader can find the full details and additional explanations in the first author, Diekert and Elder's paper \cite{eqns_free_grps}. In the next subsection we will refer to this algorithm as the CDE algorithm.

Let $F(A)$ be the free group generated by an inverse-closed set $A$ (see Notation \ref{not:IC}), and let $(M(A), \bar{ \ } \ )$ be the free monoid with involution over $A$, defined as follows. An \emph{involution} on the set $A$ is a mapping $x \mapsto \overline{x}$ such that
$\overline{\overline{x}} = x$ for all $x\in A$, which extends to an involution on the free monoid over $A$ by applying $\overline{xy}=\overline{y}\,\overline{x}$. Thus $(M(A), \bar{ \ } \ )$ is simply the free monoid on $A$ with the addition of an involution that plays the role of the inverse in the free group $F(A)$, but where no cancellations are allowed.

\textbf{Preprocessing.}
The algorithm proceeds with transforming the input equation (in the free group) into an equation in a free monoid with involution, as shown in Figure \ref{fig:preprocessing}.
 The first step of the preprocessing is to transform an equation $U=1$ in $F(A)$ over a set of variables $\mathcal{X}=\{X_1, \dots, X_k\}$ into
a \emph{triangular} system $\Delta_{F(A)}$ of equations $U_j=V_j$, $1\leq j \leq s$ for some $s>1$, by introducing new variables; that is, all equations in $\Delta_{F(A)}$ satisfy $|U_j| =2$ and $|V_j| =  1$ for all $j$, and the initial variables $X_1, \dots, X_k$ of $U=1$ are a subset of the set of variables occurring in the system $U_j=V_j$ (see \cite[Section 4]{eqns_free_grps} for details).

In the second step we leave the group setting: we produce a triangular system $\Delta_{M(A)}$ of equations $U'_j=V'_j$, $1\leq j \leq s'$ for some $s'>1$, over the free monoid with involution $(M(A), \bar{ \ } \ )$ from the triangular free group system $\Delta_{F(A)}$. This is done by using the fact that in the free group $F(A)$, a triangular equation such as $XY=Z$ (where $X, Y, Z$ are variables or constants) has a solution in reduced words over $A$ if and only if there exist words $P,Q,R$  with $X=PQ, Y=Q^{-1}R, Z=PR$ where no cancellation occurs between $P$ and $Q$, $Q^{-1}$ and $R$, and $P$ and $R$ (see \cite[Lemma 4.1]{eqns_free_grps}). At this point we introduce the rational constraint that the solutions to the system $\Delta_{M(A)}$ in $M(A)$ are freely reduced words, in the sense that no $a \bar{a}$ or $\bar{a}a$ is allowed.

Finally, the triangular system $\Delta_{M(A)}$ can be fully encoded into a single equation $U'=V'$ over $(M(A), \bar{ \ }\ )$, such that the solutions to the initial $U=1$ in the free group can be obtained from solutions to $U'=V'$ (see \cite[Section 4]{eqns_free_grps}). To do this, we use the marker $\#$ (with $\# \notin A$ and $\overline{\#}=\#$) to define a single word equation
$U'=V'$ with $U',V' \in (A \cup \{\#\})^*$, where
\begin{align}\label{eq:oneweA}
 U'= U'_1\# \cdots \# U'_{s'} &\; \text{ and }
 V'= V'_1\# \cdots \# V'_{s'}.
\end{align}
Adding the constraint that no solution of $U'=V'$ contains $\#$, the solutions of $U'=V'$ agree with the solutions of the system $U'_j=V'_j$ which, together with the constraint that no $a_i \bar{a_i}$ and $\bar{a_i}a_i$, where $a_i \in A$, are allowed, lead to solutions of $U_j=V_j$, and therefore solutions of $U=1$ as reduced words. The set of variables in $U'=V'$ contains those variables appearing in $U=1$, so transferring solutions from $U'=V'$ to $U=1$ refers to the solutions restricted to the appropriate set of variables, that is, those appearing in $U=1$ only.

 From now on we focus on the single word equation $U'=V'$ in $(M(A), \bar{ \ } \ )$.
We add to $U'$ and $V'$ the letters $a_i$ from the generating set $A$ and the variables $X_i$ appearing in $U'=V'$, all separated by $\#$, and write the equation together with this data as a single word $\Winit$: $$\Winit= \#X_1\#\dots \#X_\ell \# U'\# V'\#\overline {U'}\# \overline {V'}\# \overline{X_\ell}\#\dots \#\overline{X_1}\#.$$
An easy calculation shows that although $|\Winit| > |U|$, $\Winit$ is still linear in the size of $U$; the length of $\Winit$ will be taken as the size of the input (see \cite[Section 3.1]{eqns_free_grps} for more details on $\Winit$). The equation $U'=V'$ has solutions if and only if $\Winit=\overline{\Winit}$ has a solution. The change from the equation $U'=V'$ to the word $\Winit$ by adding  variables allows us to keep track of the transformations applied to the equation and the variables concurrently, so we not only obtain an EDT0L language from the seed word $\Winit$, but also automatically keep track of how the variables are changed through the entire algorithm.

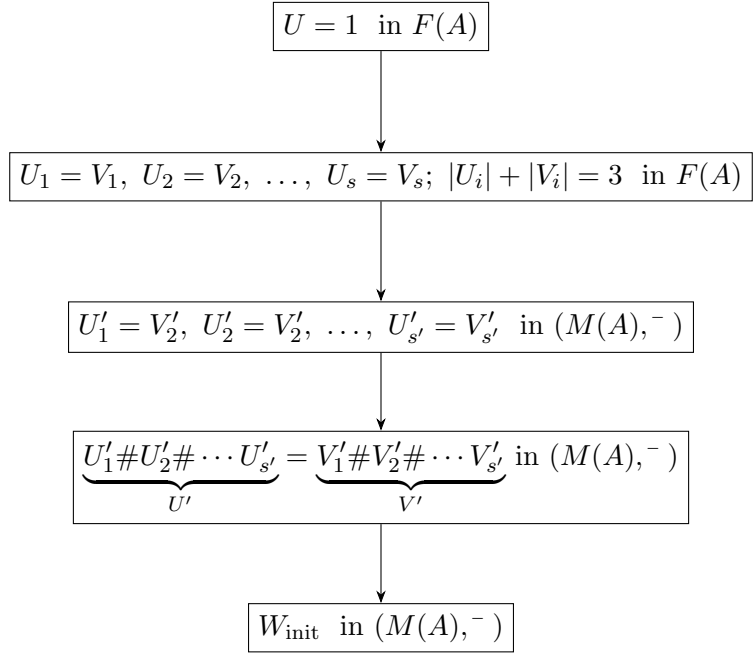
\begin{figure}

  \begin{tikzpicture}
    \node[draw, rectangle] (n1) at (0, 8) {\(U = 1\) \ in \(F(A)\)};
    \node[draw, rectangle] (n2) at (0, 6)
      {\(U_1 = V_1, \ U_2 = V_2, \ \ldots, \ U_s = V_s; \ |U_i| + |V_i| = 3\) \ in
      \(F(A)\)};
    \node[draw, rectangle] (n3) at (0, 4)
      {\(U_1' = V_2', \ U_2' = V_2', \ \ldots, \ U_{s'}' = V_{s'}'\)
      \ in \((M(A), \bar{ \ } \ )\)};
    \node[draw, rectangle] (n4) at (0, 2)
      {\(\underbrace{U_1' \# U_2' \# \cdots U_{s'}'}_{U'} =
      \underbrace{V_1' \# V_2' \# \cdots V_{s'}'}_{V'} \)
      in \((M(A), \bar{ \ } \ )\)};
    \node[draw, rectangle] (n5) at (0, 0) {\(W_\text{init}\)
    \  in \((M(A), \bar{ \ } \ )\)};

    \draw[-Stealth] (n1) to (n2);
    \draw[-Stealth] (n2) to (n3);
    \draw[-Stealth] (n3) to (n4);
    \draw[-Stealth] (n4) to (n5);
  \end{tikzpicture}

  \caption{Preprocessing diagram}
  \label{fig:preprocessing}
\end{figure}

%where For the main result one needs, for technical reasons, that for every variable $X_{i}$ in the range $1\leq i \leq k$ the variable  $\#X_i\#$ appears in the initial equation.

\textbf{Main algorithm.} We construct a finite directed graph (that is, a non-deterministic finite-state automaton) $\cG$ whose edges are labelled by monoid homomorphisms: this FSA will act as the rational control applied to the seed word $\Winit$.  The homomorphisms are over the free monoid with involution $M(C, \ \bar{ \ } \ )$, where $C$ is
 an extended alphabet containing $A$; this allows us to introduce new letters (and remove them as well) when substitutions and compressions are employed. The size of $C$ is bounded in terms of $|A|$.

 \textbf{Getting the solutions.}  Once the graph $\cG$ has been constructed, the crucial (and initially counterintuitive) point is that in order to produce solutions, our algorithm
follows paths backwards from final to initial states, using the transformations labelling the arcs from
 the last to the first one in order to produce the solutions. Thus the direction of each arc used to get the solutions is opposite to that of the homomorphism labelling the arc when the graph was constructed.

\textbf{The vertices of $\cG$.} Each vertex (state) of the graph contains, roughly, the following data: (i) a variation of $\Winit$, that is, a word $W$ of \emph{bounded length} (specifically, $|W|\leq 40|\Winit|$) obtained by applying transformations to $\Winit$, (ii) information about the current alphabet (which is a dynamic subset $B$ of $C$: that is, $B$ depends on the particular state) and variables present in $W$, (iii) the rational constraints imposed on the variables in $W$ (the constraints are implemented as homomorphisms to a finite monoid), and (iv) `types' of variables and letters that help with compressing high powers of a letter in $W$ (see \cite[Definitions 6.1, 6.2]{eqns_free_grps}). The initial state of the automaton $\cG$ corresponds to $\Winit$ together with the initialisation of the rational constraints, and the accept states correspond to any bounded length word $W$ obtained from $\Winit$ such that $W=\overline{W}$ and $W$ does not contain any variables (plus some additional technical requirements used for book keeping \cite[Section 6.2]{eqns_free_grps}).

A key fact is that the graph $\cG$ is finite: the lengths of the words $W$ in the states are bounded in terms of the length of the input $\Winit$, and thus there are only finitely many states.

\textbf{The edges of $\cG$.} In terms of transitions (or edges) between states, there are two main kinds: \emph{substitution arcs} and \emph{compression arcs}. Both edge kinds are labeled by endomorphisms of $M(C, \ \bar{ \ } \ )$.
%With a few exceptions, all labels are non-erasing endomorphisms, that is, for each map/label $h$ and each word $w \in C^*$ we have $|w| \leq |(w)h|$.
The main kind of substitution arc is the most intuitive, where $X \rightarrow aX$ or $X \rightarrow Xa$, that is, where a variable $X$ is replaced by $aX$ or $Xa$, with $a\in C$. In certain limited situations substitutions of the form $X \rightarrow 1$ (removing variable $X$), $X \rightarrow X X'$ (splitting variable $X$) or `preparing' for compression by allowing a variable to commute with a letter $a$ (see the meaning of `types' in \cite[Section 3.3]{eqns_free_grps}), are used.
The compression arcs transform the constants, but do not affect the variables. The main kind are \emph{pair} and \emph{block} compression arcs, but renaming of letters and reduction of the alphabet when no information is lost, are also allowed. The pair compression arcs are labeled by maps of the form $h(c)=u$, where $|u|=2$, so sending some letter to a pair of distinct letters, and the block compression arcs have the form $h(c)=c^2$ or $h(c)=ac$, where $c$ is a power of $a$.\

 In a nutshell, considering that the solutions are obtained by following the arcs described above in reverse orientation, we encounter the following transformations:
   \begin{itemize}

   %(nondeterministically)
      \item[--]  \emph{pop} variables: $X\ra aX$, $X\ra Xa$, or $X \ra \text 1$
         (last one only if the constraint allows it)

      \item[--]  \emph{compress pairs} of constants $ab\ra c$ where $c$ is a new constant.  %So we need a larger alphabet $C\supset A$. %We define $\rho(c)=\rho(ab)$.

   \item[--]  \emph{compress blocks} of letters $aa\dots a\ra a_\ell$ where $a_\ell$ is a new constant.

  \item[--] Eventually, substitute all variables $X\ra \text 1$, $Y\ra  \text 1$ etc, so that two words just in constants remains. If these words are identical we accept the new state together with the transformation, otherwise we discard it.

    %and $\overline X\ra \overline X\overline a$
   % --- we redefine $\rho$ on the new $X$ so that   $\rho(X)=\rho(a)\rho'(X)$  (and if this makes $\rho'(X)=0$ we reject this guess)
   \end{itemize}
Note that the first move (pop) increases the length of the equation, but gets us closer to a solution.
 The two compression moves will reduce the length of the equation while enlarging the set of constants. However, we enforce bounds on the sizes of the equation and the current alphabet at each state, so the increase in length or size of alphabet is controlled.

The automaton $\cG$ is constructed in a non-deterministic manner, by guessing the constraints at every vertex and the type of label for an edge: it produces vertices satisfying the definition of an allowable vertex (size of $W$ bounded, appropriate constraints etc), and checks which types of arcs are suitable between a given pair of vertices. Although $\cG$ is of exponential size in terms of $\Winit$, during the algorithm the working space only holds pairs of vertices, together with the arcs between them, which leads to the low space complexity for the algorithm, that is, \NSPACE$(n\log n)$ (see \cite[Section 3.8]{eqns_free_grps}). A very careful trimming of the automaton is also performed.

\begin{ex} A simplified algorithm to solve the equation $aXXba=YYaX$ over $A=\{a,b\}$ is:
\begin{itemize}
\item[--] Start by {\em guessing} the first and last letters of the variables, and pop/substitute:

\medskip %them

$X\ra Xa$ \hspace{1cm}  $aXaXaba=YYaXa$

$X\ra b X$ \hspace{1cm}  $abX abX aba=YYabXa$

$Y\ra aY$ \hspace{1cm}  $abX abX aba=aYaYabXa$

$Y\ra Yb$ \hspace{1cm}  $abX abX aba=aYbaYbabXa$

\medskip

%\item[--] We get pairs of constants in between variables that repeat, such as $ab$. %(the number of variables never increases).
%
%\item[--] We {\em compress} these pairs, to decrease the length of the equation.

\item[--] Substitute $X\ra \text 1$  and $Y\ra  \text 1$, so that two words just in constants remains. If these words are identical we use the new vertex as an accept state, otherwise we discard it.

\end{itemize}
Recall that the algorithm works with $\Winit$, which has the form $$\#X\#Y\#aXXba\#YYaX\#\overline{aXXba}\#\overline{YYaX}\#\overline{Y}\#\overline{X}\#.$$
Suppose that we denote the first pop by $\tau_1\colon X \ra Xa$, the second by $\tau_2\colon X \ra bX$, the third by $\tau_3\colon Y \ra aY$, and the fourth by $\tau_4\colon Y \ra Yb$. The  pops transform $X$ into $bXa$ and $Y$ into $aYb$, so the current equation starts with $\#bXa\#aYb \dots$. If we then apply $\tau_5 \colon X\ra \text 1, Y\ra  \text 1$ the current word becomes $\#ba\#ab \#abababa\#abababa\# \dots$, and we have obtained equal words in the middle (with no variables involved). Finally we use a compression $f \colon ba \ra c_1, ab \ra c_2$. Let $C=A \cup \{c_1, c_2\}.$

Now denoting all the maps in the opposite sense of $\tau_i$, $1\leq i \leq 5$, by $h_i$, we note that they are the identity on $C^*$, and the reverse of $f$ is $h_6 \colon c_1 \ra ba\), \(c_2 \ra ab$.
We get the solutions for $X$, $Y$ by applying the maps $h_1, \dots, h_6$ to the seed $\#c_1\#c_2\#\dots$, that is, $(c_1)h_6  \dots h_1$, $(c_2)h_6  \dots h_1$ and get that $X=ba, Y=ab$ satisfy the initial equation.
\end{ex}

The challenging part, once the graph $\cG$ has been constructed, is to show that all solutions are obtained from it, and no incorrect ones have been introduced in the process.
One uses probabilistic arguments to show that the bounds we impose on the equation and alphabet sizes are not restrictive, and that we indeed obtain all solutions from $\cG$ (see \cite[Lemma 3.34]{eqns_free_grps}).

\subsection{Sketch of the algorithm for torsion-free hyperbolic groups.}

%\subsubsection*{Reduction from torsion-free hyperbolic to free groups: an overview.}

Let $G$ be a torsion-free hyperbolic group with a finite, inverse-closed generating set $S$, let $F(S)$ be the free group on $S$, and let $\pi$ be the natural projection from $F(S)$ to $G$. Given constants $\lambda, \mu \geq 0$, we will say that a word is a $(G, S,\lambda, \mu)$-quasi-geodesic if it is a word over $S$ that represents a $(\lambda, \mu)$-quasigeodesic in $G$.

Take as input an equation of length $n$ over a set of variables $\mathcal{X}=\{X_1, \dots, X_k\}$. As in the free group case, we first transform the input equation into a triangular system $\Phi$ of equations in $G$, that is, a system where each equation has length $3$. This is done by introducing new variables, and generates $\Oh(n)$ equations (see Section \ref{sec:freegp} and \cite[Section 4]{eqns_free_grps} for details).

 From now on assume that the system $\Phi$ consists of $q\in \Oh(n)$ equations of the form  $X_jY_j=Z_j$, where $1\leq j \leq q$. In a hyperbolic group, the direct reduction of a triangular equation to a system of cancellation-free equations, done previously from the free group to the free monoid with involution, is no longer possible. Instead of looking for geodesic solutions $(g_1,g_2,g_3)$ in $G$ to an equation $XY=Z$, represented by a geodesic triangle in the Cayley graph of $G$ as on the left side of Figure \ref{fig:triangle}, one finds the solutions $g_i$ via their canonical representatives $\theta(g_i)$.
 Canonical representatives of elements in torsion-free hyperbolic groups were defined by Rips and Sela in \cite{rips_sela}, and we refer the reader to \cite{rips_sela} for their construction, and to \cite[Appendix C]{CE2019} for a basic exposition.

Let $g\in G$. For a fixed integer $T\geq 1$, called the \textit{criterion}, the canonical representative $(g)\theta_T$ of $g$ with respect to $T$ is a $(G,S,\lambda, \mu)$-quasi-geodesic word over $S$ which satisfies $(g)\theta_T =_G g$, $(g\theta_T)^{-1}=(g^{-1})\theta_T$. If $T$ is well-chosen, a number of combinatorial stability properties make canonical representatives with respect to $T$ particularly suitable for solving triangular equations in hyperbolic groups.
It is essential for the language characterisations in our main results that canonical representatives with respect to $T$ are $(G,S,\lambda, \mu)$-quasi-geodesics, with $\lambda$ and $\mu$ depending only on $\delta$, and not on $T$. We therefore write $(g)\theta$ instead of $(g)\theta_T$.

By \cite{rips_sela} there exist $y_i, c_i \in F(S)$ such that the equations $(g_1)\theta=y_1c_1y_2$, $(g_2)\theta=y_2^{-1}c_2y_3$, $(g_3)\theta=y_3^{-1}c_3y_1^{-1}$ are cancellation-free (no cancellation occurs between $y_1$ and $c_1$, $c_1$ and $y_2$ etc.) equations over $F(S)$. Moreover, the $y_i$s should be viewed as `long' prefixes and suffixes that coincide, and the $c_1 c_2 c_3$ as a `small' inner circle with circumference in $\mathcal{O}(n)$, like in Figure \ref{fig:triangle}.

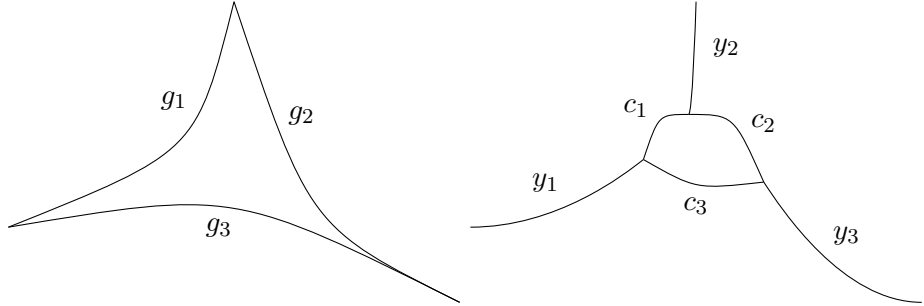
\begin{figure}[ht]
    \centering
   %  \begin{subfigure}[t]{0.45\textwidth}
        \begin{tikzpicture}[font=\sffamily]
\draw[black] (0,0) .. controls  (2.5,1) ..  (3,3) ;
\draw[black]  (0,0) .. controls  (3,.5) .. (6,-1);
\draw[black]  (6,-1) .. controls  (4,0) .. (3,3);

\draw (2.2,1.3) node [label=$g_1$]  {};
\draw   (3.9,1.1) node [label=$g_2$]{};
\draw   (2.8,-.4) node [label=$g_3$] {};

\end{tikzpicture}
   %     \subcaption{Using geodesics}\label{subfig:geod}
  %  \end{subfigure}\hspace{5mm}
    %   \begin{subfigure}[t]{0.45\textwidth}
               \begin{tikzpicture}[font=\sffamily]
\draw (0,0)  parabola (2.3,.9);
\draw (6,-1)   parabola (3.9,.6);
\draw (2.9,1.5) parabola  (3,3);

\draw (2.3,.9)  .. controls (3,.5) .. (3.9,.6);
\draw (2.9,1.5)   .. controls (3.5,1.5) ..(3.9,.6);
\draw (2.3,.9)   .. controls (2.5,1.5) ..(2.9,1.5);

\draw (2.2,1.2) node [label=$c_1$]  {};
\draw   (3.9,1) node [label=$c_2$]{};
\draw   (3,-.1) node [label=$c_3$] {};

\draw (1,.2) node [label=$y_1$]  {};
\draw   (3.4,2) node [label=$y_2$]{};
\draw   (5,-.5) node [label=$y_3$] {};

\end{tikzpicture}

   %   \subcaption{Using canonical representatives}\label{subfig:can}
  %  \end{subfigure}
    \caption{Solutions to $XY=Z$ in a hyperbolic group.}\label{fig:triangle}
\end{figure}

\noindent\textbf{Lifting $\Phi$ to the free group $F(S)$.}
Let $b$ be a constant that depends on $\delta$ and linearly on $q$ as in \cite{dahmani}. We run in lexicographic order through all possible tuples of words $$\mathbf c=(c_{11},c_{12},c_{13},\dots, c_{q1},c_{q2},c_{q3})$$ with $c_{ji}\in S^*$ and $|c_{ji}|_S\leq b$.

 Then for each tuple $\mathbf c$ we use Dehn's algorithm to check whether $c_{j1}c_{j2}c_{j3}=_G1$; if this holds for all $ j $ we construct a system $\Phi_{\mathbf c}$ of equations of the form
\begin{equation}\label{csystem}
X_j=P_jc_{j1}Q_j, \ Y_j=Q_j^{-1}c_{j2}R_j, \ Z_j=P_jc_{j3}R_j, \ \ 1\leq j\leq q,
\end{equation}
where $P_i, Q_i, Z_i$ are new variables. This new system has size $\Oh(n^2)$.
 In order to avoid an exponential size complexity we write down each system $\Phi_{\mathbf c}$ one at a time, so the space required for this step is $\Oh(n^2)$.
Let $\mathcal{Y}\supset \mathcal X$, $|\mathcal{Y}|=m$, be the new set of variables, including all the $P_i, Q_i, Z_i$.

Note that any solution to  $\Phi_{\mathbf c}$
in the free group $F(S)$ is clearly a solution to $\Phi$ in the original hyperbolic group $G$ when restricting to the original variables $\mathcal X$. That is, if $(w_1, \dots, w_m) \subseteq F^m(S)$ is a solution to $\Phi_{\mathbf c}$, then $((w_1)\pi, \dots, (w_k)\pi)$ is a solution to $\Phi$ in $G$. This shows the soundness of the algorithm.

Now let $\lambda_G \geq 1, \mu_G\geq 0$ be the constants provided by \cite[Proposition 3.4]{dahmani}.

%Note that if a word $w\in S^*$ is a $(G,S,\lambda_G,\mu_G)$-\qgeod\ and satisfies $w=_G1$ then $|w|_S \leq \mu_G$. We can construct an NFA $\mathcal N$ which accepts all words in $S^*$ equal to $1$ in the hyperbolic group $G$ of length at most $\mu_G$  in constant space (using for example Dehn's algorithm), and in our next step we will use this rational constraint to handle
%the variable $X_N$ added in Step 1 (to remove inequalities).

\medskip

\noindent \textbf{Solving equations with constraints in $F(S)$.}
We now run  the algorithm for solving free group equations with rational constraints from \cite{eqns_free_grps} (which we refer to as tCDE algorithm) which takes input $\Phi_\mathbf{c}$ of size $\Oh(n^2)$, plus for each $Y\in \mathcal Y$ the rational constraint that the solution for  $Y$ is a $(G,S,\lambda_G, \mu_G)$-quasi-geodesic.
Since these constraints have constant size (depending only on $G$, not the system $\Phi$),  they do not contribute to the $\Oh(n^2)$ size of the input to the CDE algorithm.

We modify the CDE algorithm to ensure every node printed by the algorithm includes the additional label $\mathbf c$. (This ensures the FSA we print for each system $\Phi_{\mathbf c}$ is distinct.)
This does not affect the complexity since $\mathbf c$ has size in $\Oh(n^2)$. Let $\mathcal C$ be the set of all tuples ${\mathbf c}$.
We run the modified CDE algorithm described above  (adding label $\mathbf c$ for each state printed) to print an FSA (possibly empty) for each $\Phi_{\mathbf c}$, which is the rational control for an EDT0L grammar that produces all solutions as freely reduced words in $F(S)$, which correspond to solutions as $(G,S,\lambda_G,\mu_G)$-quasi-geodesics to the same system $\Phi_{\mathbf c}$ in the hyperbolic group.

Note that if $(g_1,\dots, g_k)\in  G^k$ is a solution to $\Phi$ in the original hyperbolic group,  \cite{rips_sela}
guarantees that there exist canonical representatives $w_i$ that are $(G,S,\lambda_G, \mu_G)$-quasi-geodesics with $w_i=_Gg_i$, which have freely reduced forms $u_i=_Gw_i$ for $1\leq i\leq m$, and our construction is guaranteed to capture any such collection of words. More specifically, if $(w_1,\dots, w_r)$ is a solution in canonical representatives to $\Phi$
then $(u_1,\dots, u_k,\dots u_{|\mathcal Y|})$  will be included in the solution to $ \Phi_{\mathbf c}$ produced by the CDE algorithm, with $u_i$ the reduced forms of $w_i$ for $1\leq i\leq m$. This shows completeness once we union the grammars from all systems $\Phi_{\mathbf c}$ together.

To produce the solutions in $G$, we add a new start state with edges to each of the start states of the FSAs with label $\mathbf c$, for all $\mathbf c \in \mathcal C$. We obtain an automaton which is the rational control for the language $L$ of solutions to the input equation. This rational control makes $L$ an EDT0L language of $(G, S, \lambda, \mu)$-quasi-geodesics, where $\lambda:=\lambda_G, \mu:=\mu_G$.
The space needed is exactly that of the CDE algorithm for free groups on input $\Oh(n^2)$, which is $\NSPACE(n^2\log n)$.

Further steps are required to produce the solutions in terms of shortlex representatives (see \cite[Theorem 7.4(2)]{eqns_hyp_grps}), but these do not affect the language or space complexity.

\section{Virtually abelian groups}\label{sec:virtab}

The class of virtually abelian groups is special among the groups discussed in this paper, as a stronger statement about solution sets can be made. This is due to the proximity of virtually abelian groups to free abelian groups, where linear algebra techniques can be employed, and the structure of solutions has several equivalent algebraic, geometric and combinatorial descriptions.

More specifically, we describe solution sets in virtually abelian groups as $m$-regular languages, which is strictly stronger than describing them as EDT0L. The  definition of $m$-regular languages is similar to that of regular languages seen in Section 2.3. Let $A$ be a set. We write $A^n$ for the cartesian product $A \times \cdots \times A$ of $n$ copies of $A$, and write $\varepsilon$ for the empty word.

\begin{dfn}\label{def:fsa}
	An \emph{(asynchronous), $n$-variable finite-state automaton ($n$-FSA)} $\mathcal{A}$ is a \(4\)-tuple $(\Sigma,\Gamma,s_0,F)$, where
	\begin{enumerate}
		\item $\Sigma$ is a finite alphabet,
		\item $\Gamma$ is a finite directed graph with edges labelled by elements of $(\Sigma\cup \{\varepsilon\})^n$ which have at most one non-$\varepsilon$ entry,
		%\item no two edges with the same label start at a single vertex,
		\item $s_0\in V(\Gamma)$ is a chosen start vertex/state,
		\item $F\subseteq V(\Gamma)$ is a set of accept or final vertices/states.
	\end{enumerate}
	An $n$-tuple $\textbf{w}\in(\Sigma^*)^n$ is accepted by $\mathcal{A}$ if there is a directed path in $\Gamma$ from $s_0$ to $s \in F$ such that $\textbf{w}$ is obtained by reading the labels on the path, and deleting all occurrences of $\varepsilon$ in each coordinate. A language accepted by an $n$-FSA will be called \emph{$n$-regular} for short.
\end{dfn}
\begin{rmk}\label{rem:n-reg}\leavevmode
\begin{enumerate}
\item[(1)] Note that Definition \ref{def:fsa} gives a non-deterministic FSA, in that (i) edges labelled by tuples consisting entirely of $\epsilon$-coordinates are allowed, as are (ii) two edges with the same label starting at a single vertex. If we do not allow (i) and (ii) we have a \emph{deterministic} $n$-FSA.
\item[(2)] Definition \ref{def:fsa} coincides with the usual definition of a regular language when $n=1$. However, for $n>1$, the classes of languages accepted by deterministic and non-deterministic $n$-FSA do not coincide. See \cite[Exercise 2.10.3]{groups_langs_aut} for an example.
\end{enumerate}
\end{rmk}

We collect the characterisations of solution sets in virtually abelian groups below.

\begin{theorem}\label{thm:vab}
Let $G$ be a finitely generated virtually abelian group. Then a solution set over $G$
\begin{enumerate}
\item is a rational subset of $G^n$, for some $n\geq 1$ \cite[Theorem 1.1 (2)]{CE2022};
\item can be described by a set of geodesic representatives that form an $m$-regular, for some $m\geq 1$, and therefore EDT0L language, over any generating set of $G$ \cite[Theorem 1.3 (4)]{CE2022};
\item can be described, in terms of the natural normal form over the set of generators $X=\Sigma \cup T$ (where $\Sigma$ is a symmetric set of generators for a finite index free abelian subgroup of $G$, and $T$ is a finite transversal set for this subgroup), as an $m$-regular, and therefore EDT0L language. An EDT0L
system for this EDT0L language is
constructible in \(\ns(n^2)\) (Corollary 3.12 and Proposition 3.15 of \cite{VAEP} or Theorem 1.4 (4) of \cite{CE2022}).

\end{enumerate}

\end{theorem}

To our knowledge, virtually abelian groups are the only groups where solution sets are rational.

%\begin{ex}\label{ex:fsa}
%The language $L=\{(u,v) \mid u,v \in \{a,b\}^*, u, v \textrm{ have the same number of } a\}$ is $2$-regular and given by the $2$-fsa in Figure \ref{fig:2fsa}. %Here and later in the paper, we indicate the start state with an arrow, and the accept state(s) with a double circle.
%\end{ex}
%\begin{figure}[h!]
%		\centering
%		\includegraphics[width=0.3\textwidth]{FSAexample.eps}
%		\caption{$2$-variable automaton for Example \ref{ex:fsa}}
%		\label{fig:2fsa}
%	\end{figure}

The following proposition, which is similar to Lemma 2.20(1) of \cite{VAEP}, provides the bridge between $m$-regular and EDT0L languages that appears in the statements in Theorem \ref{thm:vab} (2) and (3).
% \begin{proposition}\label{prop:forgetful}
% 	Let $L\subset(\Sigma^*)^n$ be an $n$-regular language.
% \begin{itemize}
% \item[(i)] If $\theta\colon(\Sigma^*)^n\to\Sigma^*;~(w_1,w_2,\ldots,w_n)\mapsto w_1w_2\cdots w_n$ is the map that `forgets' the tuple structure, then $\theta(L)\subset\Sigma^*$ is EDT0L.

% \item[(ii)]	Similarly, if $\theta_{\#}\colon(\Sigma^*)^n\to(\Sigma\cup {\#})^*;(w_1,w_2,\ldots,w_n)\mapsto w_1\#w_2\#\cdots \#w_n$ is the map that inserts $\#$ between the tuple coordinates, then $\theta_{\#}(L)\subset (\Sigma \cup {\#})^*$ is EDT0L.
% \end{itemize}
% \end{proposition}
\begin{proposition}\label{prop:forgetful}
	Let $L\subset(\Sigma^*)^n$ be an $n$-regular language. If $\theta_{\#}\colon(\Sigma^*)^n\to(\Sigma\cup \{\#\})^*$ is the map $(w_1,w_2,\ldots,w_n)\mapsto w_1\#^{\epsilon_1} w_2\#^{\epsilon_2} \cdots \#^{\epsilon_n} w_n$,
 where \(\epsilon_i \in \{0, \ 1\}\) for all \(i\), that inserts a $\#$ between some of the tuple coordinates, then $(L)\theta_{\#}\subset (\Sigma \cup \{\#\})^*$ is EDT0L.
\end{proposition}
\begin{rmk}
	By considering the case \(n=1\), it is clear that the converse of Proposition \ref{prop:forgetful} does not hold (since there are EDT0L languages which are not regular), and so being \(n\)-regular is a strictly stronger property than being EDT0L.
\end{rmk}

\section{The Heisenberg group and Diophantine equations over the ring of integers}\label{sec:Heis}

  \begin{dfn}
    The \textit{Heisenberg group} \(H(\mathbb{Z})\) is the class \(2\)
    nilpotent group defined by the presentation
    \[
      H(\mathbb{Z}) = \langle a, \ b, \ c \mid c = [a, \ b], \  [a, \ c] =
      [b, \ c] = 1 \rangle.
    \]
    The generator \(c\) is redundant, however we include it as it is usually
    easier to work with the generating set \(\{a, \ b, \ c\}\) than \(\{a, \
    b\}\).

    The \textit{Mal'cev generating set} for the Heisenberg group is the set
    \(\{a, \ b, \ c\}\).
  \end{dfn}

  \begin{dfn}
    The \textit{Mal'cev normal form} for the Heisenberg group is the normal
    form that maps an element \(g \in H(\mathbb{Z})\) to the unique word of
    the form \(a^i b^j c^k\), where \(i, \ j, \ k \in \mathbb{Z}\), that
    represents \(g\).
  \end{dfn}

  The fact that this is a (unique) normal form for \(H(\mathbb{Z})\) is
  well-known (see, for example \cite{DuchinLiangShapiro}).

  Whilst it is undecidable whether a system of equations in the Heisenberg group
  admits a solution \cite{DuchinLiangShapiro}, there is an algorithm that
  determines whether a single equation admits a solution
  \cite{DuchinLiangShapiro}. Even in class \(2\) nilpotent groups in general,
  the satisfiability of single equations is undecidable \cite{Romankov2016}, so
  single equations in the Heisenberg group are close to the `boundary' of
  decidability. It remains open whether the solution language to a single
  equation in the Heisenberg group is EDT0L. In the one-variable case, however,
  there is a positive answer:

  \begin{theorem}[\cite{EDT0L_heisenberg}]
    \label{Heisenberg_eqns_EDT0L_thm}
    Let \(L\) be the solution language to a single equation with one
    variable in the Heisenberg group, with respect to the Mal'cev generating
    set and normal form. Then
    \begin{enumerate}
      \item The language \(L\) is EDT0L;
      \item An EDT0L system for \(L\) is constructible in \(\ns(
      n^8(\log n)^2)\).
    \end{enumerate}
  \end{theorem}

  The proof of Theorem \ref{Heisenberg_eqns_EDT0L_thm} starts by reducing
  the question to two-variable quadratic equations in the ring of integers.
  The majority of the proof is then to show that these equations have EDT0L
  solution languages.
%\lcc{Alex, can you actually write down (state) a quadratic equation below? It will stand out more, and then you can refer to that equation when you mention \(\mathcal{E}\) in the proof.}
  \begin{theorem}
    \label{quad_eqn_thm}
    The solution language to a two-variable quadratic equation \(\mathcal E\):
    \[
    \alpha X^2 +
    \beta XY + \gamma Y^2 + \delta X + \epsilon Y + \zeta = 0
    \]
    in integers is
    EDT0L, accepted by an EDT0L
    system that is constructible in \(\ns(
    n^4 \log n)\), with the input size taken to be \(\max(|\alpha|, |\beta|,
    |\gamma|, |\delta|, |\epsilon|, |\zeta|)\).
  \end{theorem}

  Proving Theorem \ref{quad_eqn_thm} uses a method of Lagrange to reduce
  arbitrary two-variable quadratic equations to \textit{Pell's equation}, which is
  an equation in the ring of integers of the form
  \[
    X^2 - DY^2 = 1,
  \]
  where \(X\) and \(Y\) are variables, and \(D > 0\) is not a perfect square.
  We insist that \(D > 0\) as if \(D \leq 0\), this has finitely many solutions,
  and we insist \(D\) is not a perfect square, as in this case we can factorise
  the equation to show that we can again only have finitely many solutions.
  To do this, it must first be reduced to the \textit{generalised Pell's
  equation}, which is of the form
  \[
    X^2 - DY^2 = N,
  \]
  where \(N \in \mathbb{Z}\). Lagrange's method shows that the set of
  solutions to an arbitary two-variable quadratic equation \(\mathcal E\)
  in integers is equal to
  \[
    \left\{ \left. \left(\frac{\alpha x + \beta y + \gamma}{\delta}, \
      \frac{x - \lambda}{\mu} \right) \ \right| \
    (x, \ y) \text{ is a solution to } X^2 - DY^2 = N \right\},
  \]
  where \(N, \ D \in \mathbb{Z}\) are computable from \(\mathcal{E}\). The
  proof that solutions to Pell's equation are EDT0L is straightforward, and
  we include the proof below. This can be easily modified for the generalised
  Pell's equation. What is more difficult is showing that the linear combinations
  of the solutions (that is, integers of the form \(\alpha x + \beta y + \gamma\)
  for \(\alpha, \ \beta, \ \gamma \in \mathbb{Z}\)) form an EDT0L language.
  The `division' part of the reduction is done using a construction of the
  EDT0L system, and does not require number theoretic methods.

  We give the proof that the solutions to any Pell's equation form an EDT0L
  language below.

  \begin{lem}
    Let \(X^2 - DY^2 = 1\) be an instance of Pell's equation, where
    \(D \in \mathbb{Z}_{> 0}\) is not a perfect square. Then
    \[
      \{a^x \# a^y \mid x^2 - Dy^2 = 1\}
    \]
    is an EDT0L language over the alphabet \(\{a, \ a^{-1}, \ \#\}\).
  \end{lem}

  \begin{proof}
    The solutions to Pell's equation \(X^2 - DY^2 = 1\) take the form
    \[
      \{(x_n, \ y_n) \mid n \in \mathbb{Z}_{\geq 0}\},
    \]
    where \((x_0, \ y_0) = (1, \ 0)\), \((x_1, \ y_1)\) is the smallest
    solution to \(X^2 - DY^2 = 1\) excepting \((1, \ 0)\), and \((x_n, \ y_n)\)
    for \(n > 1\) is recursively defined by
    \[
      x_n = x_1 x_{n - 1} + D y_1 y_{n - 1}, \quad y_n = y_1 x_{n - 1} +
      x_1 y_{n - 1}.
    \]
    This fact is well-known (see for example Chapter 3 of \cite{quad_dioph_eqns_book}). There has been a significant
    amount of work that uses different methods to compute \((x_1, \ y_1)\) with
    varying degrees of efficiency, however, for the purposes of showing that
    the language is EDT0L, we do not need to compute this. To construct the
    EDT0L system we would, and showing that the space complexity in which the EDT0L
    system can be constructed is nice would require one of these methods. First
    note that it suffices to show that the set of non-negative solutions is
    EDT0L. To see this, we will show that
    \[
      L = \{a^x \# b^y \mid x, \ y \in \mathbb{Z}_{\geq 0}, \ x^2 - Dy^2 = 1\}
    \]
    is EDT0L. Then note that \((x, \ y)\) is a solution if and only if \((x, \
    -y)\) is a solution if and only if \((-x, \ y)\) is a solution if and only
    if \((-x, \ -y)\) is.
%    \lcc{It is confusing here to talk about monoid morphisms while using inverses ... so should we do A, B, or some other approach?}
    Thus by applying a free monoid homomorphism to map
    \((a, \ b)\) to any of \((a, \ a)\), \((a^{-1}, \ a)\), \((a, \ a^{-1})\)
    and \((a^{-1}, \ a^{-1})\), we can show that the set of solutions lying in
    any quadrant is EDT0L, using Theorem \ref{EDT0L_closure_properties_thm}. We
    can then use Theorem \ref{EDT0L_closure_properties_thm} to show that the
    union of these languages is EDT0L. As this is the set of all solutions,
    this completes the proof.

    We now define our EDT0L system \(L\). Our extended alphabet will be
    \(C = \{a, \ b, \ \#, \ a_x, \ a_y, \ b_x, \ b_y\}\) and our start word will
    be \(a_x \# b_x\). Define \(\varphi \in \End(C^\ast)\) by
    \begin{align*}
      a_x \varphi & = a_x^{x_1} a_y^{D y_1}
      & b_x \varphi = b_x^{x_1} b_y^{D y_1} & \\
      a_y \varphi & = a_x^{y_1} a_y^{x_1}
      & b_y \varphi = b_x^{y_1} b_y^{x_1}. &
    \end{align*}
    Then define \(\theta \in \End(C^\ast)\) by
    \begin{align*}
      a_x \varphi & = a
      & b_x \varphi = \varepsilon \\
      a_y \varphi & = \varepsilon
      & b_y \varphi = b.
    \end{align*}

    Our rational control is \(\varphi^\ast \theta\) (see Figure
    \ref{Pell_eqn_fig}). By construction, if \(n \in \mathbb{Z}_{\geq 0}\),
    then \((a_x \# b_x) \varphi^n\) will be of the form \(u \# v\),
    where \(u\) contains precisely
    \(x_n\) occurrences of \(a_x\) and \(y_n\) occurrences of \(a_y\), and
    \(v\) contains precisely \(x_n\) occurrences of \(b_x\) and \(y_n\) of
    \(b_y\). Thus \((a_x \# b_x) \varphi^n \theta = a^{x_n} \# b^{y_n}\),
    and we have shown that our EDT0L system accepts \(L\).

    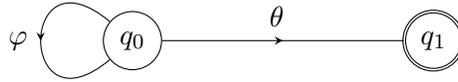
\begin{figure}
      \begin{tikzpicture}
        [scale=.8, auto=left,every node/.style={circle}]
        \tikzset{
        % style to apply some styles to each segment of a path
        on each segment/.style={
          decorate,
          decoration={
            show path construction,
            moveto code={},
            lineto code={
              \path [#1]
              (\tikzinputsegmentfirst) -- (\tikzinputsegmentlast);
            },
            curveto code={
              \path [#1] (\tikzinputsegmentfirst)
              .. controls
              (\tikzinputsegmentsupporta) and (\tikzinputsegmentsupportb)
              ..
              (\tikzinputsegmentlast);
            },
            closepath code={
              \path [#1]
              (\tikzinputsegmentfirst) -- (\tikzinputsegmentlast);
            },
          },
        },
        % style to add an arrow in the middle of a path
        mid arrow/.style={postaction={decorate,decoration={
              markings,
              mark=at position .5 with {\arrow[#1]{stealth}}
            }}},
      }

        \node[draw] (q0) at (0, 0) {\(q_0\)};
        \node[draw, double] (q1) at (5, 0)  {\(q_1\)};

        \draw[postaction={on each segment={mid arrow}}] (q0) to (q1);

        \draw[postaction={on each segment={mid arrow}}] (q0) to
        [out=140, in=-140, distance=2cm] (q0);

        \node (l1) at (-1.9, 0) {\(\varphi\)};

        \node (l2) at (2.4, 0.4) {\(\theta\)};

      \end{tikzpicture}
            \caption{Rational control for \(L = \{a^x \# b^y \mid (x, \ y) \in
      \mathbb{Z}_{\geq 0}^2 \text{ is a solution to } X^2 - DY^2 = 1\}\), with
      start state \(q_0\) and accept state \(q_1\).}
      \label{Pell_eqn_fig}
    \end{figure}
  \end{proof}

\section{Closure properties}\label{sec:closure}

  We summarise the group theoretic constructions that preserve the
  property of having EDT0L solution languages to systems of equations. It
  remains open whether this property is stable under many other
  constructions, such as passing to finite index overgroups.

  \begin{theorem}[\cite{EDT0L_extensions}, Theorem
  A]
    Let \(G\) and \(H\) be groups where solution languages to systems of
    equations are EDT0L, with respect to normal forms \(\eta_G\) and
    \(\eta_H\), respectively, and these EDT0L systems are constructible in
    \(\ns(f)\), for some \(f\). Then in the following groups, solutions to
    systems of equations     are EDT0L, and an EDT0L system can be constructed
    in non-deterministic \(f\)-space:
    \begin{enumerate}
      \item \(G \wr F\), for any finite group \(F\);
      \item \(G \times H\);
      \item Any finite index subgroup of \(G\);
    \end{enumerate}
    In the following groups, solutions to systems of equations are EDT0L, and
    an EDT0L system can be constructed in \(\ns(n^4 \log n)\):
    \begin{enumerate}[resume]
      \item Any group that is virtually a direct product of hyperbolic
      groups;
      \item Dihedral Artin groups.
    \end{enumerate}
    If \(\eta_G\) and \(\eta_H\) are both quasigeodesic or regular, then the
    same will be true for the normal forms used in (1), (2) and (3). It is
    possible to choose normal forms for the groups that are virtually direct
    products of hyperbolic groups (4), and dihedral Artin groups in (5) that
    are regular and quasigeodesic.
  \end{theorem}

  Proving (1) and (2) relies on a technical lemma about EDT0L languages where
  every word contains the same number occurrences of the letter \(\#\).
  It shows that such EDT0L systems can be generated using a well-behaved EDT0L
  system, that allows two such languages to be `concatenated in parallel', as
  is required for direct product constructions. We must first formally
  define the property that such EDT0L systems will have to allow for
  concatenation in parallel.

  \begin{dfn}
    Let \(\Sigma\) be an alphabet, and \(\#_1, \ \ldots, \ \#_n \in \Sigma\). A
    \textit{\((\#_1, \ \ldots, \ \#_n)\)-separated EDT0L system} is an EDT0L
    system \(\mathcal{H}\), with a start word of the form \(\omega_0 \#_1
    \omega_1 \#_2 \cdots \#_n \omega_n\), where \(\omega_i \in (\Sigma
    \backslash \{\#_1, \ \ldots, \ \#_n\})^\ast\) for all \(i\), and such that
    \(\#_i \phi^{-1} = \{\#_i\}\) for all \(i\) and every \(\phi\) in the
    rational control.
  \end{dfn}

  \begin{lem}[\cite{EDT0L_extensions}, Lemma 3.4]
    \label{n_hash_sep_lem}
    Let \(L\) be an EDT0L language, such that every word in \(L\) contains
    precisely \(n\) occurrences of the letter \(\#\), where \(n \in
    \mathbb{Z}_{\geq 0}\). Let \(f \colon \mathbb{Z}_{\geq 0} \to
    \mathbb{Z}_{\geq 0}\). Then
    \begin{enumerate}
      \item There is a \((\#, \ \ldots, \ \#)\)-separated EDT0L system
      \(\mathcal H\) that accepts \(L\).
      \item If an EDT0L system for \(L\) is constructible in \(\ns(f)\), then
      \(\mathcal H\) is constructible in \(\ns(f)\).
    \end{enumerate}
  \end{lem}

  Using Lemma \ref{n_hash_sep_lem}, it is straightforward to prove the
  following. This result shows the fact that direct products preserve the
  property of having EDT0L solution languages to systems of equations. Wreath
  products with finite groups require a little extra work; see
  \cite{EDT0L_extensions} for details.

    \begin{lem}[\cite{EDT0L_extensions}, Lemma 3.5]
    \label{intermesh_EDT0L_lem}
    Let \(L\) and \(M\) be EDT0L languages, such that every word in \(L \cup M\)
    contains precisely \(n\) occurrences of the letter \(\#\). Let \(f \colon
    \mathbb{Z}_{\geq 0} \to \mathbb{Z}_{\geq 0}\). Then
    \begin{enumerate}
      \item The language
      \[
        N = \{u_0 v_0 \# \cdots \# u_n v_n \ | \ u_0 \# \cdots
        \# u_n \in L, \ v_0 \# \cdots \# v_n \in M\},
      \]
      is EDT0L;
      \item If EDT0L systems for \(L\) and \(M\) are constructible in
      \(\ns(f)\), then an EDT0L system for \(N\) is constructible in \(\ns(f)\).
    \end{enumerate}
  \end{lem}

  Before we discuss the proof of the fact that finite index subgroups (of finitely generated groups) also
  preserve the property of having EDT0L solution languages, we must first
  define the generating set and normal form. By Schreier's Lemma, we know that a (finite) generating set can be explicitly constructed, and this
  generating set induces a normal form. We give the definition of Schreier generators below,
  but refer the reader to \cite{EDT0L_extensions} for a proof.

  \begin{dfn}
    Let \(G\) be a group, generated by a finite set \(\Sigma\), \(H\) be a
    finite index subgroup of \(G\), and \(T\) be a right transversal of \(H\) in
    \(G\). For each \(g \in G\), let \(\bar{g}\) be the (unique) element of
    \(T\) that lies in the coset \(Hg\). The \textit{Schreier generating set}
    for \(H\), with respect to \(T\) and \(\Sigma\), is defined to be
    \[
      Z = \{tx \overline{tx}^{-1} \mid t \in T, \ x \in \Sigma\}.
    \]
  \end{dfn}

  Given a normal form for a given group with a given finite generating
  set, the proof of Schreier's Lemma also constructs an explicit normal
  form for the finite index subgroup, using the normal form of the
  original group. We give the normal form below. Passing to this normal
  form preserves the properties of being regular and quasigeodesic.

  \begin{dfn}
    \label{Schreier_norm_form_def}
    Let \(G\) be a group, generated by a finite set \(\Sigma\), \(H\) be a
    finite index subgroup of \(G\), and \(T\) be a right transversal of \(H\)
    in \(G\).  Let \(Z\) be the Schreier generating set for \(H\). Fix a normal
    form \(\eta\) for \((G, \ \Sigma)\).

    We define the \textit{Schreier normal form} \(\zeta\) for \((H, \ Z)\),
    with respect to \(\eta\), as follows. Let \(h \in H\), and suppose \(h \eta
    = a_1 \cdots a_n\), where \(a_1, \ \ldots, \ a_n \in \Sigma^\pm\). Let
    \(t_0\) be the unique element of \(T \cap H\), and define \(t_i =
    \overline{a_1 \cdots a_i}\). Define \(h \zeta\) by
    \begin{equation}
      \label{Schreier_nrom_form_eqn}
      h \zeta = (t_0 a_1 \overline{t_0 a_1}^{-1}) (t_1 a_2 \overline{t_1
      a_2}^{-1}) \cdots (t_{n - 1} a_n \overline{t_{n - 1}a_n}^{-1}).
    \end{equation}
    The fact that this indeed defines an element of \(H\), and equals \(h\), is
   part of the proof of Schreier's Lemma.
  \end{dfn}

  The proof of the following relies on indexing the letters in the extended
  alphabet with generators of \(G\) and elements of a fixed right
  transversal, to keep track of the coset our element lies in,
  to ensure we stay within the normal form.

  \begin{proposition}[\cite{EDT0L_extensions},
  Proposition 6.3]
    \label{fin_index_EDT0L_prop}
    Let \(f \colon \mathbb{Z}_{\geq 0} \to \mathbb{Z}_{\geq 0}\). Let \(G\)
    be a group where solutions to systems of equations are EDT0L in \(\ns(f)\),
    with respect to a normal form \(\eta\). Then the same holds in any finite
    index subgroup of \(G\) with respect to the Schreier normal form, inherited
    from \(\eta\).
  \end{proposition}

\section{Baumslag-Solitar groups}
  \label{sec:BS}

Duncan, Evetts, Rees and Holt \cite{DuncanEvettsHoltRees} proved that various
equations in soluble Baumslag-Solitar groups have EDT0L languages with respect
to a standard normal form, including multiplication `triples' and centralisers
of fixed elements. A \textit{Baumslag-Solitar group} is a group defined by
the presentation
\[
  \text{BS}(j, \ k) = \langle a, \ b \mid b^{-1} a^j b = a^k \rangle,
\]
for some \(j, \ k \in \mathbb{Z} \backslash \{0\}\). These are soluble precisely
when \(j = 1\) or \(k = 1\), so we will (without loss of generality) assume \(j
= 1\). If \(k = \pm 1\), then this group is virtually abelian, and so need not be
considered here. In addition, the authors in \cite{DuncanEvettsHoltRees} assume
\(k > 1\), although they conjecture that the proofs will hold up when \(k <
-1\), modulo some technicalities.

We now define the normal form that is used to show solution languages to certain
equations are EDT0L.

\begin{lem}[\cite{DuncanEvettsHoltRees}, Section 3]
\label{lem:BSnormform}
  Every element in \(\text{BS}(1, \ k)\) is uniquely expressible in the form
  \[
    b^t \alpha^{i_m} b^{-1} \alpha^{i_m - 1} b^{-1} \cdots \alpha^{i_1} b^{-1}
    \alpha^s,
  \]
  where \(t \in \mathbb{Z}\), \(m, \ s \in \mathbb{Z}_{\geq 0}\), \(i_j \in
  \{0, \ \ldots, \ k - 1\}\) for all \(j\), with \(i_m \neq 0\) whenever
  \(m > 0\).
\end{lem}

\begin{theorem}[\cite{DuncanEvettsHoltRees}]
  Let \(G = B(1, \ k)\) for some \(k\). Then, all with respect to the
  normal form in Lemma \ref{lem:BSnormform},
  \begin{enumerate}
    \item The solution language to \(XY = Z\) is EDT0L;
    \item The solution language to \(XYZ = 1\) is EDT0L;
    \item If \(g \in G\), then the solution language to \(Xg = gX\) is EDT0L.
  \end{enumerate}
\end{theorem}

With respect to this normal form, Duncan, Evetts, Holt and Rees conjecture that
there exist equations with solution languages that are not EDT0L. This does not
necessarily mean that there will be other regular normal forms for which all
solutions will be EDT0L; however, since a group with a decidable Diophantine
problem, but non-EDT0L solution languages with respect to any `sensible' normal
form has yet to be found, this would be significant progress in using EDT0L
languages to understand solutions to equations. They conjecture that the
solution language to the equation
\[
 XY = YX
\]
is not EDT0L. By intersecting this solution language with the regular
language \(b^\ast a \# b^\ast a^\ast\), showing that this solution
language is not EDT0L is reduced to showing that the language
\[
  \{b^r a \# b^{rn}a^\frac{k^{rn} - 1}{k^r - 1} \mid r, \ n \geq 0\}
\]
is not EDT0L.

\section*{Acknowledgements}
The authors would like to thank George Metcalfe and the  Universität Bern for their hospitality and support during the final part of writing this survey. The first named author acknowledges a Scientific Exchanges grant (number IZSEZ0$\_ $213937) of the Swiss National Science Foundation, and the second named author thanks the Heilbronn Institute for Mathematical Research for support during this work.

Finally, this survey would not have been written without the essential contributions to the area of many mathematicians and valued collaborators: Volker Diekert, Andrew Duncan, Murray Elder, Alex Evetts, Derek Holt, Sarah Rees, to name just a few.

\bibliography{references}

% \bib, bibdiv, biblist are defined by the amsrefs package.
\begin{bibdiv}
\begin{biblist}

\bib{quad_dioph_eqns_book}{book}{
      author={Andreescu, Titu},
      author={Andrica, Dorin},
       title={Quadratic {D}iophantine equations},
      series={Developments in Mathematics},
   publisher={Springer, New York},
        date={2015},
      volume={40},
        ISBN={978-0-387-35156-8; 978-0-387-54109-9},
         url={https://doi.org/10.1007/978-0-387-54109-9},
        note={With a foreword by Preda Mih\u{a}ilescu},
      review={\MR{3362224}},
}

\bib{Asveld1977}{book}{
      author={Asveld, {Peter R.J.}},
       title={A characterization of {ET0L} and {EDT0L} languages},
    language={English},
      series={Memorandum / Department of Applied Mathematics},
   publisher={University of Twente, Department of Applied Mathematics},
        date={1976},
      number={129},
}

\bib{bishop_elder_journal}{incollection}{
      author={Bishop, Alex},
      author={Elder, Murray},
       title={Bounded automata groups are co-{ET}0{L}},
        date={2019},
   booktitle={Language and automata theory and applications},
      series={Lecture Notes in Comput. Sci.},
      volume={11417},
   publisher={Springer, Cham},
       pages={82\ndash 94},
         url={https://doi.org/10.1007/978-3-030-13435-8_6},
      review={\MR{3927687}},
}

\bib{EDT0L_permuations}{article}{
      author={Brough, Tara},
      author={Ciobanu, Laura},
      author={Elder, Murray},
      author={Zetzsche, Georg},
       title={Permutations of context-free, {ET}0{L} and indexed languages},
        date={2016},
     journal={Discrete Math. Theor. Comput. Sci.},
      volume={17},
      number={3},
       pages={167\ndash 178},
      review={\MR{3529277}},
}

\bib{eqns_free_grps}{article}{
      author={Ciobanu, L.},
      author={Diekert, V.},
      author={Elder, M.},
       title={Solution sets for equations over free groups are {EDT}0{L}
  languages},
        date={2016},
        ISSN={0218-1967},
     journal={Internat. J. Algebra Comput.},
      volume={26},
      number={5},
       pages={843\ndash 886},
         url={https://doi.org/10.1142/S0218196716500363},
      review={\MR{3536438}},
}

\bib{eqns_hyp_grps_conf}{incollection}{
      author={Ciobanu, L.},
      author={Elder, M.},
       title={Solutions sets to systems of equations in hyperbolic groups are
  {EDT}0{L} in {PSPACE}},
        date={2019},
   booktitle={46th {I}nternational {C}olloquium on {A}utomata, {L}anguages, and
  {P}rogramming},
      series={LIPIcs. Leibniz Int. Proc. Inform.},
      volume={132},
   publisher={Schloss Dagstuhl. Leibniz-Zent. Inform., Wadern},
       pages={Art. No. 110, 15},
      review={\MR{3984927}},
}

\bib{eqns_hyp_grps}{article}{
      author={Ciobanu, L.},
      author={Elder, M.},
       title={The complexity of solution sets to equations in hyperbolic
  groups},
        date={2021},
        ISSN={0021-2172},
     journal={Israel J. Math.},
      volume={245},
      number={2},
       pages={869\ndash 920},
  url={https://doi-org.manchester.idm.oclc.org/10.1007/s11856-021-2232-z},
      review={\MR{4358266}},
}

\bib{appl_L_systems_GT}{article}{
      author={Ciobanu, L.},
      author={Elder, M.},
      author={Ferov, M.},
       title={Applications of {L} systems to group theory},
        date={2018},
        ISSN={0218-1967},
     journal={Internat. J. Algebra Comput.},
      volume={28},
      number={2},
       pages={309\ndash 329},
         url={https://doi.org/10.1142/S0218196718500145},
      review={\MR{3786421}},
}

\bib{CE2019}{article}{
      author={Ciobanu, Laura},
      author={Elder, Murray},
       title={Solutions sets to systems of equations in hyperbolic groups are
  edt0l in pspace},
        date={2019},
     journal={arXiv e-prints},
        note={arXiv:1902.07349},
}

\bib{CE2022}{article}{
      author={Ciobanu, Laura},
      author={Evetts, Alex},
       title={Rational sets in virtually abelian groups: languages and growth},
        date={2022},
     journal={arXiv e-prints},
        note={arXiv:2205.05621},
}

\bib{dahmani}{article}{
      author={Dahmani, F.},
       title={Existential questions in (relatively) hyperbolic groups},
        date={2009},
        ISSN={0021-2172},
     journal={Israel J. Math.},
      volume={173},
       pages={91\ndash 124},
         url={https://doi.org/10.1007/s11856-009-0084-z},
      review={\MR{2570661}},
}

\bib{dahmani_guirardel}{article}{
      author={Dahmani, Fran\c{c}ois},
      author={Guirardel, Vincent},
       title={Foliations for solving equations in groups: free, virtually free,
  and hyperbolic groups},
        date={2010},
        ISSN={1753-8416},
     journal={J. Topol.},
      volume={3},
      number={2},
       pages={343\ndash 404},
         url={https://doi.org/10.1112/jtopol/jtq010},
      review={\MR{2651364}},
}

\bib{more_than_1700}{incollection}{
      author={Diekert, V.},
       title={More than 1700 years of word equations},
        date={2015},
   booktitle={Algebraic informatics},
      series={Lecture Notes in Comput. Sci.},
      volume={9270},
   publisher={Springer, Cham},
       pages={22\ndash 28},
         url={https://doi.org/10.1007/978-3-319-23021-4_2},
      review={\MR{3448027}},
}

\bib{DEijac}{article}{
      author={Diekert, V.},
      author={Elder, M.},
       title={Solutions to twisted word equations and equations in virtually
  free groups},
        date={2020},
        ISSN={0218-1967},
     journal={Internat. J. Algebra Comput.},
      volume={30},
      number={4},
       pages={731\ndash 819},
         url={https://doi.org/10.1142/S0218196720500198},
      review={\MR{4113851}},
}

\bib{EDT0L_RAAGs}{incollection}{
      author={Diekert, V.},
      author={Je\.{z}, A.},
      author={Kufleitner, M.},
       title={Solutions of word equations over partially commutative
  structures},
        date={2016},
   booktitle={43rd {I}nternational {C}olloquium on {A}utomata, {L}anguages, and
  {P}rogramming},
      series={LIPIcs. Leibniz Int. Proc. Inform.},
      volume={55},
   publisher={Schloss Dagstuhl. Leibniz-Zent. Inform., Wadern},
       pages={Art. No. 127, 14},
      review={\MR{3577188}},
}

\bib{DiekertMuscholl}{article}{
      author={Diekert, Volker},
      author={Muscholl, Anca},
       title={Solvability of equations in graph groups is decidable},
        date={2006},
        ISSN={0218-1967},
     journal={Internat. J. Algebra Comput.},
      volume={16},
      number={6},
       pages={1047\ndash 1069},
         url={https://doi.org/10.1142/S0218196706003372},
      review={\MR{2286422}},
}

\bib{DuchinLiangShapiro}{article}{
      author={Duchin, Moon},
      author={Liang, Hao},
      author={Shapiro, Michael},
       title={Equations in nilpotent groups},
        date={2015},
        ISSN={0002-9939},
     journal={Proc. Amer. Math. Soc.},
      volume={143},
      number={11},
       pages={4723\ndash 4731},
         url={https://doi.org/10.1090/proc/12630},
      review={\MR{3391031}},
}

\bib{DuncanEvettsHoltRees}{article}{
      author={Duncan, Andrew},
      author={Evetts, Alex},
      author={Holt, Derek~F.},
      author={Rees, Sarah},
       title={Using {EDT0L} systems to solve some equations in the solvable
  {B}aumslag-{S}olitar groups},
        date={2023},
     journal={arXiv e-prints},
        note={arXiv:2204.13758},
}

\bib{VAEP}{article}{
      author={Evetts, Alex},
      author={Levine, Alex},
       title={Equations in virtually abelian groups: languages and growth},
        date={2022},
        ISSN={0218-1967},
     journal={Internat. J. Algebra Comput.},
      volume={32},
      number={3},
       pages={411\ndash 442},
         url={https://doi.org/10.1142/S0218196722500205},
      review={\MR{4417480}},
}

\bib{FerteMarinSenizergues14}{article}{
      author={Fert\'{e}, Julien},
      author={Marin, Nathalie},
      author={S\'{e}nizergues, G\'{e}raud},
       title={Word-mappings of level 2},
        date={2014},
        ISSN={1432-4350},
     journal={Theory Comput. Syst.},
      volume={54},
      number={1},
       pages={111\ndash 148},
         url={https://doi.org/10.1007/s00224-013-9489-5},
      review={\MR{3149083}},
}

\bib{GrigorchukKurchanov}{incollection}{
      author={Grigorchuk, R.~I.},
      author={Kurchanov, P.~F.},
       title={Some questions of group theory related to geometry},
        date={1993},
   booktitle={Algebra, {VII}},
      series={Encyclopaedia Math. Sci.},
      volume={58},
   publisher={Springer, Berlin},
       pages={167\ndash 232, 233\ndash 240},
        note={Translated from the Russian by P. M. Cohn},
      review={\MR{1265271}},
}

\bib{herbst_thomas}{article}{
      author={Herbst, Thomas},
      author={Thomas, Richard~M.},
       title={Group presentations, formal languages and characterizations of
  one-counter groups},
        date={1993},
        ISSN={0304-3975},
     journal={Theoret. Comput. Sci.},
      volume={112},
      number={2},
       pages={187\ndash 213},
         url={https://doi.org/10.1016/0304-3975(93)90018-O},
      review={\MR{1216320}},
}

\bib{groups_langs_aut}{book}{
      author={Holt, D.~F.},
      author={Rees, S.},
      author={R\"{o}ver, C.~E.},
       title={Groups, languages and automata},
      series={London Mathematical Society Student Texts},
   publisher={Cambridge University Press, Cambridge},
        date={2017},
      volume={88},
        ISBN={978-1-316-60652-0; 978-1-107-15235-9},
         url={https://doi.org/10.1017/9781316588246},
      review={\MR{3616303}},
}

\bib{diophantine_metabelian_grps}{article}{
      author={Kharlampovich, Olga},
      author={L\'{o}pez, Laura},
      author={Myasnikov, Alexei},
       title={The {D}iophantine problem in some metabelian groups},
        date={2020},
        ISSN={0025-5718},
     journal={Math. Comp.},
      volume={89},
      number={325},
       pages={2507\ndash 2519},
         url={https://doi.org/10.1090/mcom/3533},
      review={\MR{4109575}},
}

\bib{EDT0L_heisenberg}{article}{
      author={Levine, Alex},
       title={Formal languages, quadratic {D}iophantine equations and the
  {H}eisenberg group},
        date={2022},
     journal={arXiv e-prints},
        note={arXiv:2203.04849},
}

\bib{EDT0L_extensions}{article}{
      author={Levine, Alex},
       title={E{DT}0{L} solutions to equations in group extensions},
        date={2023},
        ISSN={0021-8693},
     journal={J. Algebra},
      volume={619},
       pages={860\ndash 899},
         url={https://doi.org/10.1016/j.jalgebra.2022.11.031},
      review={\MR{4534696}},
}

\bib{Makanin_systems}{article}{
      author={Makanin, G.~S.},
       title={Systems of equations in free groups},
        date={1972},
        ISSN={0037-4474},
     journal={Sibirsk. Mat. \v{Z}.},
      volume={13},
       pages={587\ndash 595},
      review={\MR{0318314}},
}

\bib{Makanin_semigroups}{article}{
      author={Makanin, G.~S.},
       title={The problem of the solvability of equations in a free semigroup},
        date={1977},
     journal={Mat. Sb. (N.S.)},
      volume={103(145)},
      number={2},
       pages={147\ndash 236, 319},
      review={\MR{0470107}},
}

\bib{Makanin_eqns_free_group}{article}{
      author={Makanin, G.~S.},
       title={Equations in a free group},
        date={1982},
        ISSN={0373-2436},
     journal={Izv. Akad. Nauk SSSR Ser. Mat.},
      volume={46},
      number={6},
       pages={1199\ndash 1273, 1344},
      review={\MR{682490}},
}

\bib{computational_compl}{book}{
      author={Papadimitriou, Christos~H.},
       title={Computational complexity},
   publisher={Addison-Wesley Publishing Company, Reading, MA},
        date={1994},
        ISBN={0-201-53082-1},
      review={\MR{1251285}},
}

\bib{Razborov_thesis}{thesis}{
      author={Razborov, A.~A.},
       title={On systems of equations in free groups},
        type={Ph.D. Thesis},
        date={1971},
        note={In Russian},
}

\bib{Razborov_english}{incollection}{
      author={Razborov, A.~A.},
       title={On systems of equations in free groups},
        date={1995},
   booktitle={Combinatorial and geometric group theory ({E}dinburgh, 1993)},
      series={London Math. Soc. Lecture Note Ser.},
      volume={204},
   publisher={Cambridge Univ. Press, Cambridge},
       pages={269\ndash 283},
      review={\MR{1320290}},
}

\bib{rips_sela}{article}{
      author={Rips, E.},
      author={Sela, Z.},
       title={Canonical representatives and equations in hyperbolic groups},
        date={1995},
        ISSN={0020-9910},
     journal={Invent. Math.},
      volume={120},
      number={3},
       pages={489\ndash 512},
         url={https://doi.org/10.1007/BF01241140},
      review={\MR{1334482}},
}

\bib{Romankov2016}{article}{
      author={Roman'kov, Vitaly~A.},
       title={Diophantine questions in the class of finitely generated
  nilpotent groups},
        date={2016},
        ISSN={1433-5883},
     journal={J. Group Theory},
      volume={19},
      number={3},
       pages={497\ndash 514},
         url={https://doi.org/10.1515/jgth-2016-0504},
      review={\MR{3510840}},
}

\bib{math_theory_L_systems}{book}{
      author={Rozenberg, G.},
      author={Salomaa, A.},
       title={The mathematical theory of {L} systems},
      series={Pure and Applied Mathematics},
   publisher={Academic Press, Inc. [Harcourt Brace Jovanovich, Publishers], New
  York-London},
        date={1980},
      volume={90},
        ISBN={0-12-597140-0},
      review={\MR{561711}},
}

\bib{ET0L_EDT0L_def_article}{article}{
      author={Rozenberg, Grzegorz},
       title={Extension of tabled {${\rm OL}$}-systems and languages},
        date={1973},
        ISSN={0091-7036},
     journal={Internat. J. Comput. Information Sci.},
      volume={2},
       pages={311\ndash 336},
      review={\MR{0413614}},
}

\end{biblist}
\end{bibdiv}
\bibliographystyle{abbrv}
\end{document}